\documentclass[letterpaper,leqno,11pt]{amsart}

\usepackage[T1]{fontenc}
\usepackage[utf8]{inputenc}

\usepackage[margin=1.5in]{geometry}

\usepackage{titlesec}

\usepackage{amsmath,amsthm}
\allowdisplaybreaks[1]
\usepackage{mathtools}
\usepackage{gensymb}
\usepackage{newtx}
\usepackage{mathbbol}
\usepackage{mathrsfs}
\usepackage{xfrac}

\usepackage{enumitem}

\usepackage[dvipsnames]{xcolor}
\usepackage[hidelinks]{hyperref}
\hypersetup{colorlinks,linkcolor={Sepia},citecolor={Maroon},urlcolor={BrickRed}} 
\usepackage{cleveref}
\usepackage[backend=biber,style=numeric,maxbibnames=99,]{biblatex}
\DeclareFieldFormat[misc]{title}{\mkbibquote{#1}}

\counterwithin{figure}{section}
\counterwithin{equation}{section}

\theoremstyle{definition}

\newtheorem{thm}{Theorem}[section]
\newtheorem{prop}{Proposition}[section]
\newtheorem{lemma}{Lemma}[section]

\theoremstyle{remark}
\newtheorem*{rmk}{Remark}

\titleformat{\section}{\vskip25pt\centering\normalfont\large\bf}{\thesection.}{1em}{}
\titleformat{\subsection}[runin]{\vskip10pt\normalfont\normalsize\bf}{\thesubsection.}{1em}{}

\newcommand{\rd}{\mathrm{d}}
\newcommand{\re}{\mathrm{e}}
\newcommand{\ri}{\mathrm{i}}

\newcommand{\cA}{\mathscr{A}}
\newcommand{\cC}{\mathscr{C}}
\newcommand{\cD}{\mathscr{D}}
\newcommand{\cE}{\mathscr{E}}
\newcommand{\cF}{\mathscr{F}}
\newcommand{\cH}{\mathscr{H}}

\newcommand{\cZ}{\mathscr{Z}}
\newcommand{\NN}{\mathbb{N}}
\newcommand{\ZZ}{\mathbb{Z}}

\newcommand{\RR}{\mathbb{R}}
\newcommand{\CC}{\mathbb{C}}
\newcommand{\HH}{\mathbb{H}}
\newcommand{\PP}{\mathbf{P}}
\newcommand{\TT}{\mathbb{T}}
\newcommand{\pd}{\partial}
\newcommand{\eps}{\varepsilon}
\newcommand{\Ups}{\Upsilon}

\newcommand{\fZ}{\mathfrak{Z}}
\newcommand{\1}{\mathbb{1}}

\let\Re\relax\let\Im\relax
\DeclareMathOperator{\Re}{Re}
\DeclareMathOperator{\Im}{Im}
\DeclareMathOperator{\sgn}{sgn}

\renewcommand{\hat}{\widehat}
\renewcommand{\bar}{\overline}
\newcommand{\und}{\underline}
\newcommand{\wick}[1]{\hskip-1pt:\hskip-2pt#1\hskip-2pt:\hskip-1pt}
\newcommand{\norm}[2]{\lVert #1\rVert_{#2}}
\newcommand{\bignorm}[2]{\Big\lVert #1\Big\rVert_{#2}}
\newcommand{\ja}[1]{\langle #1 \rangle} 
 
\newcommand{\innerprod}[3]{\langle #1,#2 \rangle_{#3}}

\DeclareMathOperator*{\wlim}{w-lim}
\DeclareMathOperator{\EE}{\mathbf{E}}

\title{Phase transitions for fractional $\Phi^3_d$ on the torus}
\author{Niko Nikov}

\address{
Niko Nikov\\ School of Mathematics\\
The University of Edinburgh\\
and The Maxwell Institute for the Mathematical Sciences\\ James Clerk Maxwell Building\\ The King's Buildings\\ Peter Guthrie Tait Road\\ Edinburgh\\
EH9 3FD\\
United Kingdom}

\email{N.A.Nikov@sms.ed.ac.uk}


\date{}

\begin{document}

\maketitle
\begin{abstract}
	We consider the fractional $\Phi^3_d$-measure on the $d$-dimensional torus, with Gaussian free field having inverse covariance $(1-\Delta)^\alpha$, and show a phase transition at $d=3\alpha$. More precisely, in a regular regime $d<3\alpha$, one can construct and normalise this measure, and obtain a measure which is absolutely continuous with respect to the Gaussian free field $\mu$. At $d=3\alpha$, the behaviour depends on the size $|\sigma|$ of the nonlinearity: for $|\sigma|\ll1$, the measure exists, but is singular with respect to $\mu$, whereas for $|\sigma|\gg1$, the measure is not normalisable. This generalises a result of Oh, Okamoto, and Tolomeo (2025) on the $\Phi^3_3$-measure. 
\end{abstract}
%

\section{Introduction}
In this paper, we consider the fractional $\Phi^3_d$-measure formally given by
\begin{equation} \label{gibbs formal}
	\varrho(\rd u)=\cZ^{-1}\exp\Bigl(\frac\sigma3\int_{\TT^d} u^3\, \rd x-\frac{1}{2}\int_{\TT^d} ((1-\Delta)^\frac\alpha2 u)^2\, \rd x\Bigr)\, \rd u,
\end{equation}
and find a phase transition for $\varrho$ at $d=3\alpha$. Namely, we can make sense of $\varrho$ as a probability measure for $d < 3\alpha$, and when $|\sigma| \ll 1$ for $d = 3\alpha$; if $|\sigma| \gg 1$, then $\varrho$ is not normalisable. The formal expression above has the more precise interpretation
\begin{equation*}
	\varrho(\rd u)=\cZ^{-1}\exp\Bigl(\frac\sigma3\int_{\TT^d} u^3\, \rd x\Bigr)\, \mu(\rd u),
\end{equation*}
where $\mu$ is the centred Gaussian with inverse covariance $(1-\Delta)^\alpha$ on the space of distributions $\cD'(\TT^d)$. This paper aims to continue the (measure) study in \cite[Sections 3 and 4]{OOT24}, where Oh, Okamoto, and Tolomeo progressed the program initiated by Lebowitz, Rose, and Speer in \cite{LRS88} on (non-)construction of focusing (i.e. non-defocusing) Gibbs measures. This was motivated by the study of statistical mechanics for the nonlinear Schr\"odinger equation (NLS) in one dimesion. In \cite{LRS88}, the authors considered Gibbs measures of the form
\begin{equation} \label{Gibbs for NLS}
	\varrho(\rd u) = \cZ^{-1}\exp\Bigl(\frac 1 p \int_{\TT^d}|u|^p\,\rd x-\frac12\int_{\TT^d}|\nabla u|^2\,\rd x\Bigr)\, \rd u,
\end{equation}
with $d = 1$, $p > 2$. We interpret $\varrho$ as a potential density $\exp(-V(u))$ with respect to the massless Gaussian free field $\mu(\rd u)$. However, as $p> 2$, the scaling $u \mapsto \lambda u$ indicates that the energy functional in the exponential is unbounded from above, and so the measure in \eqref{Gibbs for NLS} has no hope of being a probability. Nevertheless, mass is a conserved quantity for NLS, and so Lebowitz, Rose, and Speer suggested considering the measure with a mass ($L^2$-norm) cutoff in the form
\begin{equation} \label{Gibbs for NLS mass cutoff}
	\varrho(\rd u) = \cZ^{-1}\exp\Bigl(\frac\sigma p\int_{\TT^d}|u|^p\,\rd x-\frac12\int_{\TT^d}|\nabla u|^2\,\rd x\Bigr)\1\{M(u)\le K\}\, \rd u,
\end{equation}
where $M(u) \coloneq \int_{\TT^d}|u|^2\,\rd x$. In \cite{Bou99}, Bourgain generalised this construction, and considered the family of generalised Gibbs measures where the mass cutoff is replaced by a mass taming. These take the form
\begin{equation} \label{Gibbs for NLS mass taming}
	\varrho(\rd u) = \cZ^{-1}\exp\Bigl(\frac\sigma p\int_{\TT^d}|u|^p\,\rd x-A\Bigl(\int_{\TT^d} |u|^2\, \rd x\Bigr)^\gamma-\frac12\int_{\TT^d}|\nabla u|^2\,\rd x\Bigr) \rd u,
\end{equation}
with $A>0$, $\gamma\ge 1$.\footnote{Observe that $\1\{|\cdot| \le K\} \le \exp(AK^\gamma)\exp(-A(\cdot)^\gamma)$; this relates \eqref{Gibbs for NLS mass cutoff} and \eqref{Gibbs for NLS mass taming}.} Either construction solves the issue of the energy functional being unbounded. Indeed, by the Gagliardo-Nirenberg-Sobolev (GNS) interpolation inequality on $\RR^d$, we have
\begin{equation} \label{GNS} \tag{GNS}
	\int_{\RR^d}|u|^p\,\rd x \le C_{\rm GNS}(d,p)\Bigl(\int_{\RR^d}|\nabla u|^2\,\rd x\Bigr)^\frac{(p-2)d}{2}\Bigl(\int_{\RR^d} |u|^2\,\rd x\Bigr)^{2+\frac{(p-2)(2-d)}{2}}.
\end{equation}
This suggest that this measure is constructible whenever $\frac{(p-2)d}2<2$, i.e. $p=2+\frac4d$. This heuristic turns out to be correct in one dimension \cite{LRS88}, while the situation becomes more complicated for $d\ge 2$. When $d=2$, it was shown that the measure never exists, independently of the mass cutoff \cite{BS96}, see also \cite{OST24}. The program of constructibility has a long history, and was completed by Oh, Okamoto, and Tolomeo in \cite{OOT24}, where the authors showed that when $d = 3$, $p = 3$, there is a phase transition emerging depending on the size of $|\sigma|$. We summarise here the current state of the art in the study of measures of the form \eqref{Gibbs for NLS mass cutoff}, \eqref{Gibbs for NLS mass taming}.
\begin{thm} \label{Phi^p_d construction summary}
	(Constructibility of $\Phi^p_d$.)
	\begin{enumerate}[label=(\roman*)]
		\item ($d=1$, \cite{LRS88,Bou94,OST22,TW24}) We state the results with $\sigma=1$. We consider
		\begin{equation*}
			\varrho(\rd u) = \cZ^{-1}\exp\Bigl(\frac1p\int_\TT |u|^p\,\rd x\Bigr)\1\{M(u)\le K\}\, \mu(\rd u).
		\end{equation*}
		Then we have the following.
		\begin{enumerate}[label=(\Roman*)]
			\item (subcritical case, $2<p<6$) $\varrho$ exists as a probability for any $K>0$.
			\item (critical case, $p=6$) $\varrho$ exists as a probability if and only if $K<\norm{Q}{L^2(\RR)}^2$, where $Q$ is the optimiser for \eqref{GNS}.
			\item (supercritical case, $p>6$) $\varrho$ is not normalisable.
		\end{enumerate}
		\item 
		\begin{enumerate}[label=(\Roman*)]
			\item ($d=2$, $p=3$, \cite{Bou99}, construction due to Jaffe; see also \cite{OST24}) The measure
			\begin{equation*}
				\varrho(\rd u) = \cZ^{-1}\exp\Bigl(\frac13\int_{\TT^2}\wick{u^3}\,\rd x\Bigr)\1\{\wick{M(u)}\,\le K\}\, \mu(\rd u)
			\end{equation*}
			exists as a probability measure.
			\item ($d=2$, $p=4$, \cite{BS96,OST24}) The measure
			\begin{equation*}
				\varrho(\rd u) = \cZ^{-1}\exp\Bigl(\frac\sigma4\int_{\TT^2}\wick{|u|^4}\,\rd x\Bigr)\1\{\wick{M(u)}\,\le K\}\, \mu(\rd u)
			\end{equation*}
			does not exist as a probability measure.
		\end{enumerate}
		\item ($d=3$, \cite{OOT24}) We consider
		\begin{equation*}
			\varrho(\rd u) = \cZ^{-1}\exp\Bigl(\frac\sigma3\int_{\TT^3}\wick{u^3}\,\rd x-A\Big|\int_{\TT^3}\wick{u^2}\,\rd x\Big|^\gamma\Bigr)\, \mu(\rd u).
		\end{equation*}
		Then we have the following.
		\begin{enumerate}[label=(\Roman*)]
			\item (weakly nonlinear case, $|\sigma|\ll1$) Following a second renormalisation, we can make sense of $\varrho$ as a probability measure of the form
			\begin{equation*}
				\varrho(\rd u) = \cZ^{-1}\exp\Bigl(\frac\sigma p\int_{\TT^3}\wick{u^3}\,\rd x-A\Big|\int_{\TT^3}\wick{u^2}\,\rd x\Big|^\gamma-\infty\Bigr)\, \mu(\rd u).
			\end{equation*}
			but $\varrho\perp\mu$.
			\item (strongly nonlinear case, $|\sigma|\gg1$) There exists a $\sigma$-finite version $\varrho^\delta$ of $\varrho$ as above, and $\varrho^\delta$ has infinite mass.
		\end{enumerate}
	\end{enumerate}
\end{thm}
Note that, in dimension three, the critical nonlinearity is $p=3$ as opposed to $p=\frac{10}3$ as predicted by our earlier heuristic. Similar measures have been studied in \cite{OOT25,GOTT24,Seo24b,LOZ24,CLO24}. \vspace{\baselineskip}

In this paper we generalise the result in \cite{OOT24} to consider non-integer dimension. While this is not a well-defined concept, in the context of stochastic quantisation, one of the standard ways of performing this generalisation is to replace the kinetic energy $\int |\nabla u|^2$ with $\int |(-\Delta)^{\frac\alpha2}u|^2$ (see, for instance, \cite{Las00,BMS03,CMW23,DGR24}), which, after performing the mass taming, leads to measures of the form \eqref{gibbs formal}.

In particular, we consider the measure $\varrho$, formally given as in \eqref{gibbs formal}, and we show that the phase transition observed in the case $d=3$, $\alpha=1$ is actually a particular case of a more general phenomenon. More specifically, the main result of this paper is the following (see Theorem~\ref{main} for a more precise statement).
%
\begin{thm} \label{mainvague}
	Assume $d\le3\alpha$. There exist nonlinearity thresholds $0<\sigma_0\le\sigma_1$ and taming parameters $A,\gamma$ for which the following is true. 
	\begin{enumerate}[label=(\roman*)]
		\item (Regular and weakly nonlinear regimes) If $d<3\alpha$ or $d=\alpha$ with $|\sigma| \le \sigma_0$, then we can construct and normalise the fractional $\Phi^3_d$-measure in the form
		\begin{equation} \label{masstamingmeasure}
			\varrho(\rd u) = \cZ^{-1} \exp\Bigl(\frac\sigma3 \int_{\TT^d} u^3\, \rd x - A\Bigl|\int_{\TT^d} u^2\, \rd x\Bigr|^\gamma\Bigr)\, \mu(\rd u),
		\end{equation}
		up to renormalisation. If $d<3\alpha$, then $\varrho \ll \mu$, and if $d=3\alpha$, then $\varrho \perp \mu$.
		\item (Critical regime, strong nonlinearity) If $d=3\alpha$ and $|\sigma| \ge \sigma_1$, the Gibbs measure is not normalisable: there exists a suitable approximation $(\varrho_N)$ of $\varrho$ such that
		\begin{equation*}
			\varrho_N(\rd u) = \cZ_N^{-1} \exp(-V_N(u))\, \mu(\rd u)
		\end{equation*}
		with $V_N(u) \to V(u)$ for $\mu$-a.e. $u$, but $\cZ_N \to \infty$. Moreover, the sequence $(\varrho_N)$ has no weak limit (even up to a subsequence) as probability measures in an appropriate space of distributions.
		\end{enumerate}
\end{thm}

The measures $\varrho$ in Theorem~\ref{mainvague} are realised as limits of  approximate (truncated) measures $\varrho_N$. In the case $d < 3\alpha$, step (1) below, together with dominated convergence, is enough to construct $\varrho$. At $d = 3\alpha$, however, we require a further renormalisation, and so $\varrho$ is realised as a weak limit of the $\varrho_N$, using a variational approach as carried out in \cite{BG20}. The proof outline in this case is as follows.
\begin{enumerate}[label=\arabic*)]
	\item (Uniform exponential integrability) Prove the following uniform boundedness of the partition functions $\cZ_N \coloneq \varrho_N(\cD'(\TT^d))$:
	\begin{equation*}
		\sup_{N\in\NN}\cZ_N\le C<\infty;
	\end{equation*}
	\item (Compactness) Prove tightness of the truncated measures $\{\varrho_N:N\in\NN\}$;
	\item (Unique limits) By (2) and Prokhorov's theorem \cite[Theorem~8.6.2]{Bog07}, any subsequence of $(\varrho_N)$ has a further subsequence which is convergent; proving uniqueness of limits allows us to conclude that the overall sequence converges to this same limit;
	\item (Singularity) Prove that $\varrho$ is mutually singular with respect to $\mu$.
\end{enumerate}
A recurring tool in proving relevant estimates is the Bou\'e-Dupuis variational formula, used as in \cite{BG20}. See Lemma~\ref{b-d} in Subsection~\ref{stochastic tools}. 
\begin{rmk}
	There is a rich literature studying the dynamical problem associated to $\Phi^p_d$-measures, which arise as invariant measures for Hamiltonian systems. In our case, for example, we can consider the following fractional stochastic damped nonlinear wave equation,
	\begin{equation} \label{stochquanteq}
		\pd_t^2 u+\pd_t u+(1-\Delta)^\alpha u-\sigma u^2 = \sqrt2\xi,
	\end{equation}
	where $\xi$ is space-time white noise. For \eqref{stochquanteq}, the measure $\varrho\otimes\mu_0$, with $\mu_0$ denoting the white noise measure, is formally invariant. By this we mean that $\Phi(t,\cdot)_\#(\varrho\otimes\mu_0) = \varrho\otimes\mu_0$ for all $t\ge 0$, where $\Phi$ denotes the flow map for \eqref{stochquanteq}. From the point of view of stochastic quantisation, equation \eqref{stochquanteq} is the canonical stochastic quantisation equation. Modulo proving local well-posedness, one can exploit similar invariance to obtain global dynamics for associated equations. See \cite{OOT24,GKOT21,GKO24} for wave dynamics, including globalisation as described above and paracontrolled arguments to prove local well-posedness. 
\end{rmk}
%

\section{Preliminaries}
In this section we collect notation to be used in the rest of the paper, as well as useful estimates. Hereon and unless otherwise stated, we write function spaces over the torus $X(\TT^d)$ as $X$. \vskip10pt

\subsection{Notation.} \label{notation subsection}
The majority of our notations will be kept consistent with \cite{OOT24}.\vspace{\baselineskip}

Subscripts in $N$ will denote frequency truncations. Our sharp frequency projection will be
\begin{equation*}
	\pi_N f(x)=\sum_{|n|_\infty\le N}\hat{f}(n)\re^{2\pi \ri n\cdot x},
\end{equation*}
with $|\cdot|_\infty$ denoting $\norm{\,\cdot\,}{\ell^\infty(\{1,\ldots,d\})}$, being particularly useful in critical regimes, as it is bounded on Lebesgue spaces. We will also have use for smooth frequency projectors. To this end, let $\phi:\RR\to[0,1]$ be a smooth bump function with support in $[-\frac{8}{5},\frac{8}{5}]$ such that $\phi=1$ on $[-\frac{5}{4},\frac{5}{4}]$; for $\xi\in\RR^d$ set $\phi_0(\xi)=\phi(|\xi|)$ and $\phi_j(\xi)=\phi(2^{-j}|\xi|)-\phi(2^{-j+1}|\xi|)$, noting that $\sum_j \phi_j=1$. Recall the Besov spaces $B^s_{p,q}$ equipped with norm
\begin{equation} \label{besovnorm}
	\norm{u}{B^s_{p,q}}=\norm{\norm{2^{sj}\phi_j(\nabla)u}{L^p}}{\ell^q_j(\NN\cup\{0\})}=\Bigl(\sum_{j=0}^\infty \norm{2^{sj}\phi_j(\nabla)u}{L^p}^q\Bigr)^{\frac{1}{q}}, 
\end{equation}
where $\phi(\nabla)$ is a Fourier multiplier with symbol $\phi$. Denote by $\cC^s$ the H\"older-Besov space $B^s_{\infty,\infty}$ and note that $H^s=B^s_{2,2}$ by Plancherel. \vspace{\baselineskip}

Fix $\alpha\in\RR$. We will denote by $\mu$ a centred Gaussian measure with covariance $(1-\Delta)^{-\alpha}$ and Cameron-Martin space $H^\alpha$, realised on distributions $\cD'$ (or $\cC^{\alpha-\frac d2-\eps}$ for any $\eps>0$). The measure $\mu$ has a series representation. Let $\xi$ be Gaussian space-time white noise on a probability space $(\Omega,\PP)$. For $n\in\ZZ^d$, let $B_n=\innerprod{\xi}{\1_{[0,t]}\re_n}{x,t}$ so that---where $\Lambda$ is the index set $\bigcup_{j=1}^d \ZZ^{j-1}\times\NN\times\{0\}^{d-j}$---we have $(B_n)_{\Lambda\cup\{0\}}$ i.i.d. and $B_{-n}=\bar{B_n}$. The $B_n$ are i.i.d. standard complex Brownian motions, i.e. $\Re\,B_n(1)\sim\Im\,B_n(1)\sim\mathscr{N}_\CC(0,\frac12)$. The cylindrical process
\begin{equation}
	Y(x,t;\omega)=\sum_{n\in\ZZ^d} \frac{B_n(t;\omega)}{\ja{n}^\alpha}\re^{2\pi in\cdot x}
\end{equation}
has ${\rm Law}\,Y(\cdot,1)=\mu$ supported in $\cC^{\alpha-\frac d2-\eps}$ for any $\eps>0$. In the regime $d\ge 2\alpha$, the random series $Y$ is typically distribution-valued, and so we cannot make sense of the powers $Y^j$. For this reason, we renormalise via Wick powers $\wick{Y^j_N}\,$, defined as $H_j(Y_N;\sigma_N)$, where $H_j$ is the $j$-th Hermite polynomial and 
\begin{equation} \label{gaussianseriesvargrowth}
	\sigma_N=\EE|Y_N(x,1)^2| \sim 
	\begin{cases} 
		N^{d-2\alpha}, & \text{if}\ d\ne 2\alpha, \\ 
		\log N, & \text{if}\ d=2\alpha
	\end{cases}
\end{equation}
independently of $x\in\TT^d$. Note that, when $d\ge 2\alpha$, one has $\sigma_N\to\infty$. See Lemma~\ref{pathwreg} for more details. Define the closure of polynomial chaoses $\cH_j$ in $L^2$ and let $\cH_{\le k}=\bigoplus_j \cH_j$. \vskip10pt

The potentials of interest will be functionals $V_N:\cD'\to\CC$ of the following forms, naturally following \eqref{masstamingmeasure}:
\begin{equation} \label{masstamingpotential}
	V_N(u) = -\frac\sigma3\int_{\TT^d} p_3(u_N)\, \rd x+A\Big|\int_{\TT^d} p_2(u_N)\, \rd x\Big|^\gamma+\beta_N, 
\end{equation}
where the $p_i$ carry Wick renormalisations where necessary, and the $\beta_N$ allow us to introduce further renormalisations. We aim to obtain $\varrho$ as in \eqref{masstamingmeasure} as a weak limit of measures
\begin{equation} \label{truncatedmeasure}
	\varrho_N(u)=\cZ_N^{-1}\exp(-V_N(u))\, \mu(\rd u), \qquad \cZ_N=\int_{\cD'} \exp(-V_N(u))\, \mu(\rd u). 
\end{equation}
\subsection{Deterministic estimates.}
One has the following well-known estimates for Sobolev and Besov spaces (see, e.g., \cite[Chapters~1 and 2]{BCD11}).
\begin{lemma}[Besov estimates] \label{besovest} The following estimates hold.
	\begin{enumerate}[label=(\alph*)]
		\item (Interpolation) Let $s,s_1,s_2\in\RR$ and $p,p_1,p_2\in\RR$ be such that $s=\theta s_1+(1-\theta)s_2$ and $p^{-1}=\theta p_1^{-1}+(1-\theta)p_2^{-1}$ for some $\theta\in[0,1]$; then
		\begin{equation} \label{sobolevinterpolation}
			\norm{u}{W^{s,p}}\lesssim \norm{u}{W^{s_1,p_1}}^\theta\norm{u}{W^{s_2,p_2}}^{1-\theta}. 
		\end{equation}
		\item (Immediate embeddings) Let $s_1,s_2\in\RR$ and $p_1,p_2,q_1,q_2\in[1,\infty]$; then
		\begin{equation} \label{immediatebesovembeddings} \begin{aligned}
			\norm{u}{B^{s_1}_{p_1,q_1}}&\lesssim \norm{u}{B^{s_2}_{p_2,q_2}}, \qquad s_1\le s_2,p_1\le p_2,q_1\ge q_2,\\
			\norm{u}{B^{s_1}_{p_1,q_1}}&\lesssim \norm{u}{B^{s_2}_{p_2,\infty}}, \qquad s_1< s_2, \\
			\norm{u}{B^0_{p_1,\infty}}&\lesssim \norm{u}{L^{p_1}}\lesssim \norm{u}{B^0_{p_1,1}}.
		\end{aligned}  \end{equation}
		Moreover the second embedding is compact.
		\item (Duality) Let $\int_{\TT^d} uv\, \rd x$ denote the Besov space duality pairing; let $s\in\RR$, and $1\le p,q\le \infty$; then
		\begin{equation} \label{besovduality}
			\Big|\int_{\TT^d} uv\, \rd x\Big| \le \norm{u}{B^s_{p,q}}\norm{v}{B^{-s}_{p',q'}}. 
		\end{equation}
		\item (Besov embedding) Let $1\le p_2\le p_1\le \infty$, $q\in[1,\infty]$, and $s_2\ge s_1+d(p_2^{-1}-p_1^{-1})$; then
		\begin{equation} \label{besovembeddings}
			\norm{u}{B^{s_1}_{p_1,q}} \lesssim \norm{u}{B^{s_2}_{p_2,q}}. 
		\end{equation}
		\item (Fractional Leibniz rule) Let $p,p_1,p_2,p_3,p_4\in[1,\infty]$ be such that $p^{-1}=p_j^{-1}+p_{j+1}^{-1}$ for $j=1,3$; then, for every $s>0$ and $q\in[1,\infty]$,
		\begin{equation} \label{besovfractionalleibniz}
			\norm{uv}{B^s_{p,q}} \lesssim \norm{u}{B^s_{p_1,q}}\norm{v}{L^{p_2}}+\norm{u}{L^{p_3}}\norm{v}{B^s_{p_4,q}}. 
		\end{equation}
	\end{enumerate}
\end{lemma}
\begin{lemma}[A Schauder estimate] \label{schauder}
	Let $(p_t)_{t>0}$ be the heat kernel. Let $s\ge 0$ and $p,q\in\RR$ have $1\le p\le q\le\infty$. Then
	\begin{equation} \label{schaudereq}
		\norm{p_t\ast u}{L^q(\TT^d)} \lesssim_{s,p,q} t^{-\frac{s}{2}-\frac{d}{2}(\frac{1}{p}-\frac{1}{q})}\norm{u}{W^{-s,p}(\TT^d)}. 
	\end{equation}
\end{lemma}
\begin{lemma}[On discrete convolutions] \label{discrconv}
	Let $a,b\in\RR$ have $a+b>d$, $a<d$. Then
	\begin{equation} \label{discrconveq}
		\sum_{m\in\ZZ^d}\frac{1}{\ja{m}^a\ja{n-m}^b}\lesssim \frac{1}{\ja{n}^{a-\lambda}} 
	\end{equation}
	for any $n\in\ZZ^d$, where $\lambda=\max\{d-b,0\}$ when $b\ne d$ and $\lambda>0$ when $b=d$ (i.e. $\lambda$ is allowed to be arbitrarily small in the latter case).
\end{lemma}
\subsection{Stochastic tools.} \label{stochastic tools}
Below we state several stochastic lemmas. The first two are properties of Hermite polynomials, while the latter is the Bou\'e-Dupuis variational formula, and will be central to our analysis in the following sections.
\begin{lemma}[Gaussian moment bound] \label{hypercontractivity}
Let $k\ge 1$ be an integer. For any $X\in \cH_{\le k}$, we have
	\begin{equation} \label{hypercontractivityeq}
		\EE|X|^p\le ((p-1)^{k}\EE|X|^2)^{\frac p2}. 
	\end{equation}
\end{lemma}
\begin{lemma}[Hermite orthogonality] \label{hermiteorthogonality}
	Let $f,g$ be jointly Gaussian with mean zero and variances $\sigma_f$, $\sigma_g$. Then, for any $k,\ell\ge 1$, we have
	\begin{equation} \label{hermiteorthogonalityeq}
		\EE[H_k(f;\sigma_f)H_\ell(g;\sigma_g)]=\delta_{k\ell}k!(\EE[fg])^k. 
	\end{equation}
\end{lemma}
\begin{lemma}[Bou\'e-Dupuis variational formula, \cite{Ust14,BD98}] \label{b-d}
	Let $\HH_{\rm a}$ be drifts, namely, prog-\allowbreak ressively-measurable processess which are $\PP$-a.s. in $L^2([0,1];L^2(\TT^d))$. Fix $N\in\NN$. Suppose that $F:C^\infty(\TT^d)\to\RR$ is measurable and such that 
	\begin{equation*}
		\EE|F(Y_N(1))|^p+\EE|\exp(-F(Y_N(1)))|^{p'}<\infty
	\end{equation*}
for some $1<p<\infty$. Then we have the following variational representation
	\begin{equation} \label{b-deq}
		-\log \EE \,\exp(-F(Y_N(1)))=\inf_{\theta\in\HH_{\rm a}}\EE\Bigl[F\Bigl(Y_N(1)+\pi_N\int_0^1\ja{\nabla}^{-\alpha}\theta(t)\, \rd t\Bigr)+\frac{1}{2}\int_0^1\norm{\theta(t)}{L^2}^2\, \rd t\Big].
	\end{equation}
\end{lemma}
\noindent Lemma~\ref{b-d} will simplify many of our calculations to come, and allow us to identify a need for a second renormalisation in the regime $d = 3\alpha$.

\subsection{Regularity estimates.}
Below is a lemma on pathwise regularity estimates for wick powers $\wick{Y_N^k(t)}$.
\begin{lemma}[Pathwise regularity of stochastic terms] \label{pathwreg} One has the following estimates.
	\begin{enumerate}[label=(\roman*)]
		\item Let $k=1,2,3$, $q\ge 2$, and $\eps>0$. Write $s=k(\alpha-\frac d2)$. Then $\wick{Y_N^k(t)}$ converges to $\wick{Y^k(t)}$ in $L^q(\Omega;\cC^{s-\eps})$ and almost-surely in $\cC^{s-\eps}$. Moreover
		\begin{equation}
			\EE\norm{\,\wick{Y_N^k(t)}\,}{\cC^{s-\eps}}^q \lesssim q^{\frac{k}{2}}<\infty
		\end{equation}
		uniformly in $N\in\NN$ and $t\in[0,1]$.
		\item Assume $d<3\alpha$. Then
		\begin{equation}
			\EE\norm{\,\wick{Y_N^2(t)}\,}{H^{-\alpha}}^2 \lesssim t^2
		\end{equation}
		uniformly in $N\in\NN$. On the other hand, assume $d = 3\alpha$. Then
		\begin{equation}
			\EE\norm{\,\wick{Y_N^2(t)}\,}{H^{-\alpha}}^2 \gtrsim t^2\log N
		\end{equation}
		for any $t\in[0,1]$.
		\item We have
		\begin{equation}
			\EE \Bigl[\int_{\TT^d} \wick{Y_N^p(1)}\, \rd x\Big]=0.
		\end{equation}
	\end{enumerate}
\end{lemma}
See \cite{GKO18} and \cite{GKO24} for proofs similar to (i). One can prove (ii) as in \cite{OOT24}. Finally, (iii) is a consequence of Hermite orthogonality. \qed

\section{(Non-)construction of $\Phi^3_d$-measure} \label{phi3d}
In this section, we focus on the (non-)construction of $\Phi^3_d$-measure in what will be based on that carried out in \cite{OOT24}. \vspace{\baselineskip}

\subsection{A change-of-variable.} \label{c-o-v}
We first discuss a change-of-variable to be used in the Bou\'e-Dupuis variational formula arising from a need for a second renormalisation in the case $d=3\alpha$. Suppose that $\beta_N=0$ for all $N$ in \eqref{masstamingpotential}. Then by Lemma~\ref{b-d}, we have
\begin{equation} \label{b-dVN} \begin{aligned}
	-&\log\int_{\cD'}\exp(-V_N(u))\, \mu(\rd u) \\
	&=\inf_{\theta\in\HH_{\rm a}}\EE\Bigl[-\frac{\sigma}{3}\int_{\TT^d}\wick{(Y_N+\Theta_N)^3}\, \rd x+A\Big|\int_{\TT^d}\wick{(Y_N+\Theta_N)^2}\, \rd x\Big|^\gamma+\frac12\int_0^1\norm{\theta(t)}{L^2}^2\, \rd t\Big],
\end{aligned} \end{equation}
where $Y_N=Y_N(1)$ and $\Theta_N=\pi_N\Theta=\pi_N\int_0^1\ja{\nabla}^{-\alpha}\theta(t)\,\rd t$. Using the binomial formula for cubic Wick powers, we investigate cross-terms $\int_{\TT^d}\wick{Y_N^j}\Theta_N^{3-j}\, \rd x$. Where it turns out for $j=0,1$ we have control (see Lemma~\ref{b-dcrossterms} below), and recalling that $\int_{\TT^d}\wick{Y_N^3}\, \rd x$ is zero under expectation, we discuss $j=2$. Using It\^o's product formula,
\begin{equation*} \begin{aligned}
	\EE&\Bigl[\int_{\TT^d}\wick{Y_N^2}\Theta_N\, \rd x\Big] \\
	&=\EE\Bigl[\int_0^1\int_{\TT^d}\wick{Y_N^2(t)}\dot\Theta_N(t)\, \rd x\, \rd t\Big]+\EE\Bigl[\int_{\TT^d}\int_0^1 \Theta_N(t)\,\rd(\,\wick{Y_N^2}\,)t\, \rd x\Big]+\EE[\,\wick{Y_N^2}\,,\Theta_N]_1 \\
	&=\EE\Bigl[\int_0^1\int_{\TT^d}\wick{Y_N^2(t)}\dot\Theta_N(t)\, \rd x\, \rd t\Big],
\end{aligned} \end{equation*}
where $\dot\Theta_N(t)=\pi_N\ja{\nabla}^{-\alpha}\theta(t)$, and $[\cdot,\cdot]$ is the bracket process. The last equality follows from the fact that $\wick{Y_N^2}$ is a martingale and $\Theta_N$ is a finite variation process. Define $\fZ^N$ by its derivative via $\fZ^N(0)=0$ and
\begin{equation}
	\dot\fZ^N(t)=\ja{\nabla}^{-2\alpha}\wick{Y_N^2(t)}
\end{equation}
and set $\fZ_N=\pi_N\fZ^N$. Then put
\begin{equation} \label{Upsdef}
	\dot\Ups^N(t)=\dot\Theta(t)-\sigma\dot\fZ_N(t)
\end{equation}
and set $\Ups_N=\pi_N\Ups^N$. One can then verify (essentially by completing the square), that
\begin{equation} \begin{aligned}
	\EE\Bigl[-\sigma\int_{\TT^d}&\wick{Y_N^2}\Theta_N\, \rd x+\frac12\int_0^1\norm{\theta(t)}{L^2}^2\, \rd t\Big] \\
	&= \EE\Bigl[\frac12\int_0^1\norm{\dot\Ups^N(t)}{H^\alpha}\, \rd t-\frac{\sigma^2}{2}\int_0^1\norm{\dot\fZ_N(t)}{H^\alpha}^2\, \rd t\Big].
\end{aligned} \end{equation}
As can be seen in Lemma~\ref{pathwreg}, the constant $\frac{\sigma^2}{2}\EE[\int_0^1\norm{\dot\fZ_N(t)}{H^\alpha}^2\, \rd t]$ appearing above is bounded uniformly in $N$ when $d<3\alpha$ and divergent when $d = 3\alpha$. In the latter regime, we perform a further renormalisation by setting $\beta_N$ equal to this diverging constant. Following the definitions above, we replace the minimisation over $\theta\in\HH_\mathrm{a}$ in \eqref{b-dVN} to minimisation over $\dot\Ups^N\in\HH^\alpha_\mathrm{a}=\ja{\nabla}^{-\alpha}\HH_\mathrm{a}$ as
\begin{equation} \begin{aligned}
	-&\log\int_{\cD'}\exp(-V_N(u))\, \mu(\rd u) \\
	&=\inf_{\dot\Ups^n\in\HH_{\rm a}^\alpha}\EE\Bigl[-\frac{\sigma}{3}\int_{\TT^d}Y_N\Theta_N^2\, \rd x-\frac\sigma3\int_{\TT^d}\Theta_N^3\, \rd x+A\Big|\int_{\TT^d}\wick{(Y_N+\Theta_N)^2}\, \rd x\Big|^\gamma \\
	&\hskip220pt+\frac12\int_0^1\norm{\dot\Ups^n(t)}{H^\alpha}^2\, \rd t\Big].
\end{aligned} \end{equation}
\subsection{The main result and proof strategy.}
Before stating our main result, we define a taming functional to be used in the strongly nonlinear regime. Let
\begin{equation} \label{anormdefn}
	\norm u\cA=\sup_{0<t\le 1} t^{\alpha-\frac d6-\eps}\norm{p_t\ast u}{L^3},
\end{equation}
(where $\eps$ will be assumed sufficiently small to close arguments), i.e. $\cA=B_{3,\infty}^{-2\alpha+\frac d3+2\eps}$. Let $s=\alpha-\frac d6-\eps$; the choice of this exponent will become clear following the proof of Proposition~\ref{uniexpint} (v). From the embedding $\cC^{\alpha-\frac d2-\eps} \hookrightarrow \cA$, it holds that $\cA$ contains the support of $\mu$. In what follows,
\begin{equation} \label{tamedpotential}
W_{N,\delta}(u)=\delta\norm{u_N}{\cA}^q+V_N(u), \qquad \vartheta_{N,\delta}(\rd u)=\cZ_{N,\delta}^{-1}\exp(-W_{N,\delta}(u))\, \mu(\rd u),
\end{equation}
where the $\cZ_{N,\delta}$ are normalisation constants and $q$ is an exponent to be chosen later. Our main result is the following.
\begin{thm}[Gibbs measure (non-)construction] \label{main}
	Assume $d\le3\alpha$. If $d<3\alpha$, set $\beta_N=0$ in the definition of $V_N$; otherwise, (i.e. $d=3\alpha$) set
	\begin{equation*}
		\beta_N=\frac{\sigma^2}2\EE\Bigl[\int_0^1\norm{\dot\fZ_N(t)}{H^\alpha}^2\, \rd t\Big].
	\end{equation*}
	Note $\beta_N\to\infty$ in this case. There exist nonlinearity thresholds $0<\sigma_0\le\sigma_1$ for which the following is true.
	\begin{enumerate}[label=(\roman*)]
		\item (Very regular regime) When $d<2\alpha$, there is a choice of $\gamma$ in $(1,2)$ such that, for any $\sigma,A$, we have the uniform exponential integrability
		\begin{equation}
			\sup_{N\in\NN}\cZ_N = \sup_{N\in\NN}\,\norm{\re^{-V_N}}{L^1(\mu)}<\infty 
		\end{equation}
		and $(\varrho_N)$ converges in total variation to the desired Gibbs measure
		\begin{equation} \label{Gibbs very regular}
			\varrho(\rd u) = \cZ^{-1}\exp\Bigl(\frac\sigma3\int_{\TT^d} u^3\, \rd x-A\Big|\int_{\TT^d}u^2\, \rd x\Big|^\gamma\Bigr)\, \mu(\rd u)
		\end{equation}
		with a finite partition function $\cZ<\infty$; here $\varrho\ll\mu$;
		\item (Regular regime) When $2\alpha\le d<3\alpha$, for any $\sigma,A$, taking $\gamma=2+\eps$ when $d=2\alpha$ and $\gamma=\frac{d}{d-2\alpha}$ when $d>2\alpha$, we have the uniform exponential integrability
		\begin{equation} \label{uniformexpintegrability}
			\sup_{N\in\NN}\,\cZ_N = \sup_{N\in\NN}\,\norm{\re^{-V_N}}{L^1(\mu)}<\infty 
		\end{equation}
		and $(\varrho_N)$ converges in total variation to the desired \textit{Wick-ordered} Gibbs measure
		\begin{equation} \label{Gibbs regular}
			\varrho(\rd u) = \cZ^{-1}\exp\Bigl(\frac\sigma3\int_{\TT^d}\wick{u^3}\, \rd x-A\Big|\int_{\TT^d}\wick{u^2}\, \rd x\Big|^\gamma\Bigr)\, \mu(\rd u)
		\end{equation}
		with a finite partition function $\cZ<\infty$; here $\varrho\ll\mu$;
		\item (Critical regime, weak nonlinearity) When $d=3\alpha$, for $0<|\sigma|<\sigma_0$, $A=A(\sigma)$ sufficiently large, and $\gamma=\frac{d}{d-2\alpha}$, we have \eqref{uniformexpintegrability} as above, and a unique weak limit $\varrho$ of $(\varrho_N)_{N\in\NN}$ (realised on $\cC^{\alpha-\frac d2-\eps}$) formally given by
		\begin{equation} \label{Gibbs critical weak nonlinear}
			\varrho(\rd u) = \cZ^{-1}\exp\Bigl(\frac\sigma3\int_{\TT^d}\wick{u^3}\, \rd x-A\Big|\int_{\TT^d}\wick{u^2}\, \rd x\Big|^3-\infty\Bigr)\, \mu(\rd u); 
		\end{equation}
moreover the limiting measure $\varrho$ is singular with respect to $\mu$;
		\item for $|\sigma|>\sigma_1$, the Gibbs measure is not normalisable in the following sense: there exist $s>0$, $q\ge 1$ such that, writing $\cA=B_{3,\infty}^{-2s}$ as in \eqref{anormdefn} and $W_{N,\delta}$ as in \eqref{tamedpotential} for any $A$ and $\gamma\ge\frac{d}{d-2\alpha}$, the measures $(\vartheta_{N,\delta})_{N\in\NN}$, $\delta>0$, given by
		\begin{equation} \label{tamedmeasure} \begin{aligned}
			\vartheta_{N,\delta}(\rd u) &= \cZ_{N,\delta}^{-1}\exp(-W_{N,\delta}(u_N))\, \mu(\rd u) \\
			&= \cZ_{N,\delta}^{-1}\exp(-\delta \norm{u_N}{\cA}^q-V_N(u))\, \mu(\rd u), 
		\end{aligned} \end{equation}
		converge weakly to a limit $\vartheta_\delta$ and
		\begin{equation} \label{sigmafinitephi33}
			\varrho^\delta(\rd u) := \exp(\delta F(u))\, \vartheta_\delta(\rd u) 
		\end{equation}
		defines a measure on $\cC^{\alpha-\frac d 2-\eps}$ with $\varrho^\delta(\cD')=\infty$; under the same assumptions, the sequence $(\varrho_N)_{N\in\NN}$ has no weak limit, even up to a subsequence, as measures on $\cA\supseteq{\rm supp}\,\mu$.
	\end{enumerate}
\end{thm}
\begin{rmk}
	The non-convergence pointed out in (iv) may not hold on a space with a weaker topology, (e.g. $\cC^{-c}$ with $c\gg1$ sufficiently large), but it does not hold on $\cA$, indicating that, even if it were to hold in some space, it would be credibly pathological.
\end{rmk}
\vskip10pt

The program for proving Part (iii) of Theorem~\ref{main} follows below:
\begin{enumerate}[label=\arabic*)]
	\item (Uniform exponential integrability) Prove the uniform boundedness
	\begin{equation*}
		\sup_{N\in\NN}\,\cZ_N\le C<\infty;
	\end{equation*}
	\item (Compactness) Prove tightness of the truncated measures $\{\varrho_N:N\in\NN\}$;
	\item (Unique limits) By (2) and Prokhorov's theorem \cite[Theorem~8.6.2]{Bog07}, any subsequence of $(\varrho_N)$ has a further subsequence which is convergent; proving uniqueness of limits allows us to conclude that the overall sequence converges to this same limit; this is the measure $\varrho$ in \eqref{Gibbs regular} and \eqref{Gibbs critical weak nonlinear} of Theorem~\ref{main};
	\item (Singularity) Prove that $\varrho$ is mutually singular with respect to $\mu$.
\end{enumerate}
Likewise we employ a similar approach to prove Part (iv) of the theorem. Here, one needs to construct a weak limit  $\vartheta_\delta$ of the measures $(\vartheta_{N,\delta})$, and prove
\begin{enumerate}[label=\arabic*)]
	\item (Well-definedness of $\varrho^\delta$) Prove that the quantity $\norm{u}\cA$ is $\vartheta_\delta$-a.s. finite; this allows us to define $\varrho^\delta$;
	\item (Non-normalisability of $\varrho^\delta$) Prove that $\varrho^\delta(\cD')=\infty$ (the approach largely follows the two-dimensional case in \cite{OST24}).
\end{enumerate}
\subsection{Construction of measures.} \label{constr}
In this subsection, we construct a limiting $\Phi^3_d$-measure $\varrho$ in the weakly nonlinear regime, and a $\sigma$-finite version of $\Phi^3_d$ via a reference measure. We first prove the uniform exponential integrability 
\begin{equation} \label{uniexpinteq}
	\sup_{N\in\NN} \cZ_N < \infty
\end{equation}
and in the case $d < 3\alpha$, construct $\varrho$ by dominated convergence; at $d = 3\alpha$ we follow the approach in \cite{OOT24}, proving tightness of the $\varrho_N$, and then using Prokhorov's theorem along with uniqueness of weak limits to obtain $\varrho$ as a weak limit.

\begin{prop}[Uniform exponential integrability] \label{uniexpint}
	Let $d,\alpha>0$ and let $V_N$ follow the definition given in Theorem~\ref{main}.
	\begin{enumerate}[label=(\roman*)]
		\item Let $d<2\alpha$. There exists some $\gamma_0$ in the interval $(1,2)$ such that, for all $\gamma\ge \gamma_0$, and any $A>0$ and $\sigma$, we have \eqref{uniexpinteq}.
		\item Let $d=2\alpha$. For any $\sigma$ and $A$ sufficiently large depending on $\sigma$ with $\gamma = 2$, or for any $\sigma,A$ with $\gamma > 2$, we have \eqref{uniexpinteq}.
		\item Let $2\alpha<d<3\alpha$. For any $\sigma,A$ with $\gamma=\frac{d}{d-2\alpha}$, we have \eqref{uniexpinteq}.
		\item Let $d=3\alpha$. There exists $\sigma_0>0$ such that, for $0<|\sigma|<\sigma_0$ and $A>0$ sufficiently large depending on $\sigma$, with $\gamma=\frac{d}{d-2\alpha}$, we have \eqref{uniexpinteq}.
		\item There exists a choice of $s>0$ and $q\in\ZZ$ for which the following is true. For any $A>0$, $\gamma\ge \frac{d}{d-2\alpha}$, $\sigma$, and $\delta>0$, we have
		\begin{equation} \label{uniexpintiii}
			\sup_{N\in\NN} \cZ_{N,\delta} < \infty.
		\end{equation}
	\end{enumerate}
\end{prop}
\begin{proof}[Proof of Proposition~\ref{uniexpint} (i).] 
	Assume $d<2\alpha$. Since $Ax^\gamma\ge Ax^{\gamma_0}-A$, it suffices to prove the result for $\gamma=\gamma_0$. By the Bou\'e-Dupuis formula, we have
	\begin{align*}
		-\log\cZ_N = \inf_{\theta\in\HH_{\rm a}}\EE\Bigl[&-\frac\sigma3\int_{\TT^d} (Y_N+\Theta_N)^3\, \rd x \\
		&+A\Big|\int_{\TT^d}(Y_N+\Theta_N)^2\, \rd x\Big|^\gamma+\frac12\int_0^1\norm{\dot\Ups^N(t)}{H^\alpha}^2\, \rd t.
	\end{align*}
	Expanding the above, we deal with each term in turn. Recall that $\EE[\int_{\TT^d} Y_N^3\, \rd x]=0$. By \eqref{besovduality}, \eqref{besovfractionalleibniz}, and Young's inequality, we have
	\begin{align*}
		\Bigl|\int_{\TT^d} Y_N^2\Theta_N\,\rd x\Bigr| &\lesssim \norm{Y_N^2}{H^{\alpha-\frac d2-2\eps}}\norm{\Theta_N}{H^{-\alpha+\frac d2+2\eps}} \\
		&\lesssim \norm{Y_N}{L^2}\norm{Y_N}{\cC^{\alpha-\frac d2-\eps}}\norm{\Theta_N}{H^{-\alpha+\frac d2+\eps}} \\
		&\lesssim C(\delta)(\norm{Y_N}{L^2}^4+\norm{Y_N}{\cC^{\alpha-\frac d2-\eps}}^4)+\delta\norm{\Theta_N}{H^\alpha}^2
	\end{align*}
	and analogously
	\begin{align*}
		\Bigl|\int_{\TT^d} Y_N\Theta_N^2\,\rd x\Bigr| &\lesssim \norm{Y_N}{\cC^{\alpha-\frac d2-\eps}}\norm{\Theta_N^2}{H^{-\alpha+\frac d2+2\eps}} \\
		&\lesssim \norm{Y_N}{\cC^{\alpha-\frac d2-\eps}}\norm{\Theta_N}{L^2}\norm{\Theta_N}{H^{-\alpha+\frac d2+2\eps}} \\
		&\lesssim C(\delta)\norm{Y_N}{\cC^{\alpha-\frac d2-\eps}}^{c(\eps)}+\delta\norm{\Theta_N}{L^2}^{2+\eps}+\delta\norm{\Theta_N}{H^\alpha}^2;
	\end{align*}
	likewise, by \eqref{besovembeddings}, \eqref{sobolevinterpolation}, and Young's inequality, we have
	\begin{align*}
		\Bigl|\int_{\TT^d} \Theta_N^3\,\rd x\Bigr| &\lesssim \norm{\Theta_N}{H^{\frac d6}} \\
		&\lesssim \norm{\Theta_N}{L^2}^{3-\frac{d}{2\alpha}}\norm{\Theta_N}{H^\alpha}^{\frac{d}{2\alpha}} \\
		&\le C(\delta)\norm{\Theta_N}{L^2}^\frac{12\alpha-2d}{4\alpha-d}+\delta\norm{\Theta_N}{H^\alpha}^2;
	\end{align*}
	moreover, 
	\begin{equation*}
		A\Bigl|\int_{\TT^d} (Y_N+\Theta_N)^2\, \rd x\Bigr|^\gamma \ge \frac A2\Bigl|\int_{\TT^d} \Theta_N^2\, \rd x\Bigr|^\gamma - C_1\Bigl(\Bigl|\int_{\TT^d} Y_N\Theta_N\, \rd x\Bigr|^\gamma - \norm{Y_N}{L^2}^\gamma\Bigr)
	\end{equation*}
	and we can argue as above to bound $|\int_{\TT^d} Y_N\Theta_N\, \rd x|^\gamma$. Using Lemma~\ref{pathwreg}, we arrive at
	\begin{align*}
		-\log\cZ_N &\ge \inf_{\theta\in\HH_\mathrm{a}} \EE\Bigl[-C_2\sigma(\delta\norm{\Theta_N}{H^\alpha}^2+\delta\norm{\Theta_N}{L^2}^{2+\eps}+C(\delta)\norm{\Theta_N}{L^2}^\frac{12\alpha-2d}{4\alpha-d} \\
		&\hskip70pt +C(\delta)(\norm{Y_N}{L^2}^4+\norm{Y_N}{L^2}^{2\gamma}+\norm{Y_N}{\cC^{\alpha-\frac d2-\eps}}^{c(\eps,\gamma)})  \\
		&\hskip70pt +\frac A2\norm{\Theta_N}{L^2}^{2\gamma}+\frac12\norm{\Theta_N}{H^\alpha}^2\Bigr] \\
		&\ge \inf_{\theta\in\HH_\mathrm{a}} \EE\Bigl[-C_2\sigma(\delta\norm{\Theta_N}{L^2}^{2+\eps}+C(\delta)\norm{\Theta_N}{L^2}^\frac{12\alpha-2d}{4\alpha-d})+\frac A2\norm{\Theta_N}{L^2}^{2\gamma}\Bigr]-C,
	\end{align*}
	and now, as $d<2\alpha$, an appropriate choice for $\gamma$ exists in the interval $(1,2)$ such that the above is bounded below uniformly in $N$. 
\end{proof}

To prove the remainder of the proposition, we will use the change-of-variable described at the beginning of Subsection~\ref{c-o-v}, and require the following lemma, estimating cross-terms which arise when using the Bou\'e-Dupuis formula as above. We delay the proof of this lemma until the end of this subsection.

\begin{lemma} \label{b-dcrossterms} Assume that $2\alpha\le d<4\alpha$. Let $\delta>0$. There exists some $\eps>0$, exponent $c\ge 1$, and constant $C(\delta)>0$ such that
	\begin{align}
		&\Big|\int_{\TT^d} Y_N\Theta_N^2\, \rd x\Big|\lesssim 1+\delta\norm{\Ups_N}{L^2}^\frac{12\alpha-2d}{4\alpha-d}+\delta\norm{\Ups_N}{H^\alpha}^{\max\{\frac{2d-4\alpha}{d},\eps\}} \label{YNTN2bd} \\
		&\hskip80pt+C(\delta)\norm{Y_N}{\cC^{\alpha-\frac d2-\eps}}^c+C(\delta)\norm{\fZ_N}{\cC^{4\alpha-d-\eps}}^c \notag \\
		&\Big|\int_{\TT^d} \Theta_N^3\, \rd x\Big| \lesssim 1+C(\delta)\norm{\Ups_N}{L^2}^\frac{12\alpha-2d}{4\alpha-d}+\delta\norm{\Ups_N}{H^\alpha}^2+\norm{\fZ_N}{\cC^{4\alpha-d-\eps}}^c \label{TN3bd}
	\end{align}
and, for all $\gamma\ge 1$,
	\begin{align}
		A\Big|\int_{\TT^d}\wick{(Y_N+\Theta_N)^2}\, \rd x\Big|^\gamma &\ge \frac{A}{2}\Big|\int_{\TT^d}(2Y_N\Ups_N+\Ups_N^2)\, \rd x\Big|^\gamma-\delta\norm{\Ups_N}{L^2}^{2\gamma} \label{tamingtermbd} \\
		&\hskip20pt-C(\delta)\Bigl(\Big|\int_{\TT^d}\wick{Y_N^2}\, \rd x\Big|^\gamma+\norm{Y_N}{\cC^{\alpha-\frac{d}{2}-\eps}}^{2\gamma}+\norm{\fZ}{\cC^{4\alpha-d-\eps}}^c\Bigr); \notag
	\end{align}
	when $d = 2\alpha$, we also have
	\begin{align}
		A\Big|\int_{\TT^d}\wick{(Y_N+\Theta_N)^2}\, \rd x\Big|^\gamma &\ge \frac{A}{2}\Big|\int_{\TT^d}\Ups_N^2\, \rd x\Big|^\gamma-\delta\norm{\Ups_N}{L^2}^{2\gamma}-\delta\norm{\Ups_N}{H^\alpha}^2 \label{tamingtermbdeasy} \\
		&\hskip20pt-C(\delta)\Bigl(\Big|\int_{\TT^d}\wick{Y_N^2}\, \rd x\Big|^\gamma+\norm{Y_N}{\cC^{\alpha-\frac{d}{2}-\eps}}^{2\gamma}+\norm{\fZ}{\cC^{4\alpha-d-\eps}}^c\Bigr). \notag
	\end{align}
\end{lemma}
\begin{proof}[Proof of Proposition~\ref{uniexpint} (ii).]
	Using the Bou\'e-Dupuis formula and our change-of-variable, we have 
	\begin{align*}
		-\log\cZ_N = \inf_{\dot\Ups_N\in\HH_{\rm a}^\alpha}\EE\Bigl[&-\sigma\int_{\TT^d}Y_N\Theta_N^2\, \rd x-\frac\sigma3\int_{\TT^d}\Theta_N^3\, \rd x \\
		&+A\Big|\int_{\TT^d}\wick{(Y_N+\Theta_N)^2}\, \rd x\Big|^\gamma \\
		&+\frac12\int_0^1\norm{\dot\Ups^N(t)}{H^\alpha}^2\, \rd t \\
		&+\Bigl(\beta_N-\int_0^1\norm{\dot\fZ_N(t)}{H^\alpha}^2\, \rd t\Bigr)\Big].
	\end{align*}
	By \eqref{tamingtermbdeasy} in Lemma~\ref{b-dcrossterms}, picking $\delta$ small enough depending on $A$, we have
	\begin{align*}
		-\log\cZ_N \ge \inf_{\dot\Ups^N\in\HH_\mathrm{a}^\alpha} \EE\Bigl[&-C_1|\sigma|(1+C(\delta)\norm{\Ups_N}{L^2}^4+\delta\norm{\Ups_N}{H^\alpha}^2) \\ 
		&+\frac A2\norm{\Ups_N}{L^2}^{2\gamma}+\frac12\int_0^1 \norm{\dot\Ups^N(t)}{H^\alpha}^2\, \rd t\Bigr].
	\end{align*}
	Once again picking $\delta$ sufficiently small (and $A$ sufficiently large where necessary) completes the proof.
\end{proof}
For the proof of Proposition~\ref{uniexpint} in the case $d=2\alpha$, it was enough to control the cubic term $\norm{\Theta_N}{L^3}^3$ in terms of $\norm{\Ups_N}{L^2}^{2\gamma}$ and $\norm{\Ups_N}{H^\alpha}^2$. This is not the case in the setting $d=3\alpha$, and so we offer the following lemma, in which we control $\norm{\Theta_N}{L^3}^3$ in terms of $|\int (2Y_N\Ups_N+\Ups_N^2)|^\gamma$ and $\norm{\Ups_N}{H^\alpha}^2$ and an additional random variable $B$ with finite moments. Again, the proof is delayed until the end of the subsection.
\begin{lemma} \label{phi33heart}
	Assume $d<4\alpha$. There exists a nonnegative random variable $B$ on $(\Omega,\PP)$ with $\EE B^p<\infty$ for all $p\ge 1$ such that
	\begin{equation} \label{phi33hearteq}
		\norm{\Ups_N}{L^2}^2 \lesssim \Big|\int_{\TT^d}(2Y_N\Ups_N+\Ups_N^2)\, \rd x\Big|+\norm{\Ups_N}{H^\alpha}^\frac{2d-4\alpha}{d}+B. 
	\end{equation}
\end{lemma}
\begin{rmk}
	We shall see from the proof of Lemma~\ref{phi33heart} and using the lemma itself, that
	\begin{equation} \label{phi33heartrmk}
		\EE\Big|\int_{\TT^d} Y_N\Ups_N\, \rd x\Big| \lesssim \EE\Bigl[\Big|\int_{\TT^d}(2Y_N\Ups_N+\Ups_N^2)\, \rd x\Big|+\norm{\Ups_N}{H^\alpha}^\frac{2d-4\alpha}{d}\Big].
	\end{equation}
	This will be of use below.
\end{rmk}

\begin{proof}[Proof of Proposition~\ref{uniexpint} (iii), (iv), (v).]
	We first prove (iii) and (iv) together. By the Bou\'e-Dupuis formula and our change-of-variable, 
	\begin{align*}
		-\log\cZ_N = \inf_{\dot\Ups_N\in\HH_{\rm a}^\alpha}\EE\Bigl[&-\sigma\int_{\TT^d}Y_N\Theta_N^2\, \rd x-\frac\sigma3\int_{\TT^d}\Theta_N^3\, \rd x \\
		&+A\Big|\int_{\TT^d}\wick{(Y_N+\Theta_N)^2}\, \rd x\Big|^\gamma \\
		&+\frac12\int_0^1\norm{\dot\Ups^N(t)}{H^\alpha}^2\, \rd t \\
		&+\Bigl(\beta_N-\int_0^1\norm{\dot\fZ_N(t)}{H^\alpha}^2\, \rd t\Bigr)\Big];
	\end{align*}
	hence we wish to find a uniform lower bound for the right-hand side of the display above. Thanks to the renormalisation via $(\beta_N)$ and Lemma~\ref{pathwreg}, the final term above is uniformly bounded under expectation. Apply Lemma \ref{b-dcrossterms}, meanwhile also using Lemma~\ref{pathwreg}, to obtain 
	\begin{equation} \label{ZNb-dbd} \begin{aligned} 
		-&\log\cZ_N \\
		&\ge \inf_{\dot\Ups^N\in\HH_{\rm a}^\alpha}\EE\Bigl[-\sigma\int_{\TT^d}Y_N\Theta_N^2\, \rd x-\frac\sigma3\int_{\TT^d}\Theta_N^3\, \rd x-\delta'\norm{\Ups_N}{L^2}^{2\gamma} \\
		&\hskip80pt +\frac A2\Big|\int_{\TT^d}(2Y_N\Ups_N+\Ups_N^2)\, \rd x\Big|^\gamma+\frac12\int_0^1\norm{\dot\Ups^N(t)}{H^\alpha}^2\, \rd t\Big]-C \\
		&\ge \inf_{\dot\Ups^N\in\HH_{\rm a}^\alpha}\EE\Bigl[-C_1|\sigma|(\delta'+C(\delta''))\norm{\Ups_N}{L^2}^\frac{12\alpha-2d}{4\alpha-d}-C_1|\sigma|(\delta'+\delta'')\norm{\Ups_N}{H^\alpha}^2-\delta'\norm{\Ups_N}{L^2}^{2\gamma} \\ 
		&\hskip80pt+\frac A2\Big|\int_{\TT^d}(2Y_N\Ups_N+\Ups_N^2)\, \rd x\Big|^\gamma+\frac12\int_0^1\norm{\dot\Ups^N(t)}{H^\alpha}^2\, \rd t\Big]-C.
	\end{aligned} \end{equation}
	In the regime $d<3\alpha$, we have $\frac{12\alpha-2d}{4\alpha-d}<2\gamma$, and so $\norm{\Ups_N}{L^2}^\frac{12\alpha-2d}{4\alpha-d}\le \delta'\norm{\Ups_N}{L^2}^{2\gamma}+C_{\delta'}$ for any $\delta'>0$: with this in mind and using Lemma~\ref{phi33heart}, the final quantity in display \eqref{ZNb-dbd} leads to
	\begin{equation} \label{ZNb-dbd1} \begin{aligned}
		-&\log\cZ_N \\
		&\ge \inf_{\dot\Ups^N\in\HH_{\rm a}^\alpha}\EE\Bigl[\Bigl(\frac A2-C_2|\sigma|(\delta'+C(\delta''))\delta'\Bigr)\Big|\int_{\TT^d}(2Y_N\Ups_N+\Ups_N^2)\, \rd x\Big|^\gamma \\
		&\hskip50pt-C_2|\sigma|((\delta'+C(\delta''))\delta'+(\delta'+\delta''))\norm{\Ups_N}{H^\alpha}^2+\frac12\int_0^1\norm{\dot\Ups^N(t)}{H^\alpha}^2\, \rd t\Big]-C;
	\end{aligned} \end{equation}
	first picking $\delta''$ based on $\sigma$, and then $\delta'$ based on $C(\delta''),\sigma$ allows us to conclude. On the other hand when $d=3\alpha$, we have $\frac{12\alpha-2d}{4\alpha-d}=2\gamma$, and so we bound the final quantity in \eqref{ZNb-dbd} like
	\begin{equation} \label{ZNb-dbd2} \begin{aligned} 
		-&\log\cZ_N \\
		&\ge \inf_{\dot\Ups^N\in\HH_{\rm a}^\alpha}\EE\Bigl[\Bigl(\frac A2-C_2|\sigma|(\delta'+C(\delta''))\Bigr)\Big|\int_{\TT^d}(2Y_N\Ups_N+\Ups_N^2)\, \rd x\Big|^\gamma \\
		&\hskip80pt-C_2|\sigma|(2\delta'+C(\delta'')+\delta'')\norm{\Ups_N}{H^\alpha}^2+\frac12\int_0^1\norm{\dot\Ups^N(t)}{H^\alpha}^2\, \rd t\Big]-C;
	\end{aligned} \end{equation}
	Note, here, that we are led to no choice but requiring $|\sigma|$ sufficiently small to achieve the required exponential integrability. This completes the proofs of (iii) and (iv). Now, we prove \eqref{uniexpintiii}. The case $d < 3\alpha$ follows from the above, so we assume $d = 3\alpha$. By the Bou\'e-Dupuis formula and our change-of-variable, we have
	\begin{align*}
		-\log\cZ_{N,\delta} = \inf_{\dot\Ups_N\in\HH_{\rm a}^\alpha}\EE\Bigl[\delta\norm{Y_N+\Theta_N}{\cA}^q&-\sigma\int_{\TT^d}Y_N\Theta_N^2\, \rd x-\frac\sigma3\int_{\TT^d}\Theta_N^3\, \rd x \\
		&+A\Big|\int_{\TT^d}\wick{(Y_N+\Theta_N)^2}\, \rd x\Big|^\gamma \\
		&+\frac12\int_0^1\norm{\dot\Ups^N(t)}{H^\alpha}^2\, \rd t \\
		&+\Bigl(\beta_N-\int_0^1\norm{\dot\fZ_N(t)}{H^\alpha}^2\, \rd t\Bigr)\Big];
	\end{align*}
	We proceed as before. If $\gamma>\frac{d-2\alpha}{d}$, first use the estimate
	\begin{equation*}
		A\Big|\int_{\TT^d}\wick{(Y_N+\Theta_N)^2}\, \rd x\Big|^\gamma \ge C_0\Big|\int_{\TT^d}\wick{(Y_N+\Theta_N)^2}\, \rd x\Big|^\frac{d-2\alpha}{d}-C_1(A,C_0)
	\end{equation*}
	for any $0<C_0<1$. By Lemma~\ref{phi33heart}, there exists a constant $C>0$ such that, for any $\delta'>0$, we have
	\begin{equation*}
		\delta'C\Big|\int_{\TT^d} (2Y_N\Ups_N+\Ups_N^2)\, \rd x\Big|^\gamma \ge \delta'\norm{\Ups_N}{L^2}^{2\gamma}-\delta'C\norm{\Ups_N}{H^\alpha}^{\frac{2d-4\alpha}{d}\gamma}-\delta'CB.
	\end{equation*}
	Now, first using \eqref{tamingtermbd} of Lemma~\ref{b-dcrossterms} with $\delta''>0$, and then using the above observation (add the right-hand side and subtract the left-hand side to obtain a lower bound), it follows that, as long as $\delta'C\le \min\{\frac A4,\frac14\}$ and $\delta''<\delta'$, we have
	\begin{align*}
		\EE&\Bigl[A\Big|\int_{\TT^d}\wick{(Y_N+\Theta_N)^2}\, \rd x\Big|^\gamma\Big] \\
		&\ge \EE\Bigl[\Bigl(\frac{A}{2}-\delta'C\Bigr)\Big|\int_{\TT^d}(2Y_N\Ups_N+\Ups_N^2)\, \rd x\Big|^\gamma+(\delta'-\delta'')\norm{\Ups_N}{L^2}^{2\gamma}-\delta'C\norm{\Ups_N}{H^\alpha}^{\frac{2d-4\alpha}{d}\gamma} \\
		&\hskip 80pt -\delta'CB-C(\sigma,\delta'')\Bigl(\Big|\int_{\TT^d}\wick{Y_N^2}\, \rd x\Big|^\gamma+\norm{Y_N}{\cC^{\alpha-\frac{d}{2}-\eps}}^{2\gamma}+\norm{\fZ}{\cC^{4\alpha-d-\eps}}^c\Bigr)\Big] \\
		&\ge \EE[C_1\norm{\Ups_N}{L^2}^{2\gamma}-C_2\norm{\Ups_N}{H^\alpha}^2]-C'
	\end{align*}
	for some constants $C_1>0$, $0<C_2\le \frac14$, and $C'>0$. (We used also Lemma~\ref{pathwreg} to control various stochastic terms.) With \eqref{YNTN2bd} of Lemma~\ref{b-dcrossterms} to control the term $\int_{\TT^d}Y_N\Theta_N^2\, \rd x$, there exists some constant $C_3>0$ such that
	\begin{align*}
		-\log\cZ_{N,\delta} &\ge \inf_{\dot\Ups^N\in\HH_{\rm a}^\alpha}\EE\Bigl[\delta\norm{Y_N+\Ups_N+\sigma\fZ_N}{\cA}^q-\frac\sigma3\int_{\TT^d}(\Ups_N+\sigma\fZ_N)^3\, \rd x \\ 
		&\hskip 150pt+C_3\norm{\Ups_N}{L^2}^{2\gamma}+C_3\norm{\Ups_N}{H^\alpha}^2\Big]-C'
	\end{align*}
	(Above, $C'>0$ has been relabelled.) Next, by Young's inequality,
	\begin{equation*}
		\delta\norm{Y_N+\Ups_N+\sigma\fZ_N}{\cA}^q \ge \frac\delta2\norm{\Ups_N}{\cA}^q-C''(\norm{Y_N}{\cA}^q+\sigma^q\norm{\fZ_N}{\cA}^q)
	\end{equation*}
	and we can estimate, using the Schauder estimate \eqref{schaudereq} and Young's convolution inequality, that
	\begin{align*}
		&\norm{Y_N}{\cA} \lesssim \sup_{0<t\le 1} t^st^{\frac\alpha2-\frac d4-\eps}\norm{Y_N}{W^{\alpha-\frac d2-\eps,3}}, \\
		&\norm{\fZ_N}{\cA} \lesssim \Bigl(\sup_{0<t\le 1}t^s\norm{p_t}{L^1}\Bigr)\norm{\fZ_N}{\cC^{4\alpha-d-\eps}};
	\end{align*}
	assuming $s>-\frac\alpha2+\frac d4$ and $d<4\alpha$, there is a choice of $\eps>0$ for which the above are finite and bounded uniformly in $N$. Moreover using H\"older and Young's inequalities, there exists an exponent $c\ge 1$ and a constant $C(\sigma)>0$ such that
	\begin{align*}
		\Big|\sigma^2\int_{\TT^d}\Ups_N^2\fZ_N\, \rd x\Big|+\Big|\sigma^3\int_{\TT^d}\Ups_N\fZ_N^2\, \rd x\Big|&\le |\sigma|^2\norm{\Ups_N}{L^2}^2\norm{\fZ_N}{L^\infty}+|\sigma|^3\norm{\Ups_N}{L^2}\norm{\fZ_N}{L^4}^2 \\
		&\le \frac{C_3}{2}\norm{\Ups_N}{L^2}^{2\gamma}+\norm{\fZ_N}{\cC^{4\alpha-d-\eps}}^c+C(\sigma).
	\end{align*}
	Here we required $2<\frac{d}{d-2\alpha}$ (i.e. $d<4\alpha$). Combining the last four displays yields, after relabelling $C'>0$,
	\begin{equation*}
		-\log\cZ_{N,\delta} \ge \inf_{\dot\Ups^N\in\HH_{\rm a}^\alpha}\EE\Bigl[\frac\delta2\norm{\Ups_N}{\cA}^q-\frac{|\sigma|}{3}\norm{\Ups_N}{L^3}^3+\frac{C_3}{2}\norm{\Ups_N}{L^2}^{2\gamma}+C_3\norm{\Ups_N}{H^\alpha}^2\Big]-C'.
	\end{equation*}
	Using Young's inequality and a Sobolev embedding, note that
	\begin{equation*}
		\norm{\Ups_N}{L^3}^3\lesssim t^{-3s}\norm{\Ups_N}{\cA}^3+\norm{\Ups_N-p_t\ast\Ups_N}{H^{\frac d6}}^3;
	\end{equation*}
	a mean-value theorem argument provides the estimate $|1-\re^{-t|n|^2}|\lesssim (t|n|^2)^\eta$ ($n\in\ZZ^d$) for any $0\le\eta\le 1$, so that
	\begin{align*}
		\norm{\Ups_N-p_t\ast\Ups_N}{H^{\frac d6}} &= \Bigl(\sum_{n\in\ZZ^d} \ja{n}^{\frac d3}|1-\re^{-t|n|^2}|^2|\hat\Ups_N(n)|^2\Bigr)^{\frac12} \\
		&\lesssim t^\eta\norm{\Ups_N}{H^{\frac{d}{6}+2\eta}}.
	\end{align*}
	Next, we will need to assume $\frac d6+2\eta\le \alpha$. It follows after an application of \eqref{immediatebesovembeddings} that there exists some $C_4>0$ such that
	\begin{equation*}
		\frac{|\sigma|}{3}\norm{\Ups_N}{L^3}^3 \le \frac{C_4|\sigma|}{3}t^{-3s}\norm{\Ups_N}{\cA}^3+\frac{C_4|\sigma|}{3}t^{3\eta}\norm{\Ups_N}{H^\alpha}^3
	\end{equation*}
	and so, choosing (randomly) $t=\frac{1}{1+\frac{4C_4|\sigma|}{3C_3}\norm{\Ups_N}{H^\alpha}}$, picking $s<2\eta$, and using Young's inequality, we find that there exists a constant $C(\sigma,\delta)>0$ such that
	\begin{align*}
		\frac{|\sigma|}{3}\norm{\Ups_N}{L^3}^3 &\le \frac{C_4|\sigma|}{3}\Bigl(1+\frac{4C_4|\sigma|}{3C_3}\norm{\Ups_N}{H^\alpha}\Bigr)^{\frac s\eta}\norm{\Ups_N}{\cA}^3+\frac{C_3}{4}\norm{\Ups_N}{H^\alpha}^2 \\
		&\le \frac\delta4\norm{\Ups_N}{\cA}^q+\frac{C_3}{2}\norm{\Ups_N}{H^\alpha}^2+C(\sigma,\delta)
	\end{align*}
	for some suitably large choice of $q$, depending on $d$, $\sigma$, and $\alpha$. Observe that overall we need $0<-\frac\alpha2+\frac d4<s<2\eta\le\alpha-\frac d6$. This completes the proof of (v) and therefore that of the proposition.
\end{proof}
\begin{rmk}
	It follows from the proof of Proposition~\ref{uniexpint} that we can pick $s=\alpha-\frac d6-\eps$. Hereon assume this to be our choice of $s$ (with $\eps$ sufficiently small to close arguments), and therefore this determines $\cA$. Using the Schauder estimate \eqref{schaudereq}, we have
\begin{align}
	\norm{u}{\cA} &=\sup_{0<t\le 1} t^s\norm{p_t\ast u}{L^3} \notag \\
	&\lesssim \sup_{0<t\le 1}t^{s-\frac{s'}2}\norm{u}{W^{-s',3}} \notag \\
	&\lesssim \norm{u}{W^{-s',3}}, \label{embedAintosobolev}
\end{align}
	as long as $s'<2s$, i.e. as long as $s'<\alpha-2\eps$. In particular we observe that $W^{\alpha-\frac d2-\eps,3}\hookrightarrow W^{-\alpha+\eps,3}\hookrightarrow\cA$, so that $\cA\supseteq{\rm supp}\,\mu$.
\end{rmk}

In the regime $d<3\alpha$, the construction of $\varrho$ does not require any additional renormalisation as described in Subsection~\ref{c-o-v}, and so Proposition~\ref{uniexpint} is sufficient to prove the strong convergence claimed in Theorem~\ref{main} (i), (ii), (iii). Naturally, the limit is either the measure in \eqref{Gibbs very regular} or \eqref{Gibbs regular}, depending on whether or not we require a Wick renormalisation. Assuming, for example, that $2\alpha\le d<3\alpha$, it suffices to prove that
\begin{equation} \label{uniexpint and dct}
	\lim_{N\to\infty} \int_{\cD'} |\exp(-V_N(u))-\exp(-V(u))|\, \mu(\rd u) = 0
\end{equation}
in order to obtain $\cZ_N \to \cZ$ and $\varrho_N \to \varrho$ in total variation. But \eqref{uniexpint and dct} is a consequence of dominated convergence together with Proposition~\ref{uniexpint}. Equally, the above applies to the reference measure $\vartheta_\delta$. \vspace{\baselineskip}

To complete the construction of $\varrho$ and $\vartheta_\delta$ in the regime $d = 3\alpha$, we proceed as in \cite{OOT24} and in Proposition~\ref{tightness} prove tightness of $\{\varrho_N:N\in\NN\}$ (resp. $\{\theta_{N,\delta}:N\in\NN\}$). Together with Proposition~\ref{uniexpint} and Prokhorov's theorem, this implies that any subsequence of $(\varrho_N)$ (resp. $(\vartheta_{N,\delta})$) has a weakly convergent subsequence. We complete the construction of $\varrho$ (resp. $\theta_\delta$) with Proposition~\ref{uniqueness}, which proves that subsequential limits are unique. To obtain a reference measure $\varrho^\delta$, we prove in Proposition~\ref{asfiniteness} that $\delta\norm{u}{\cA}$ is $\theta_\delta$-a.s. finite. Throughout the rest of the subsection, $d = 3\alpha$.

\begin{prop}[Tightness] \label{tightness}
	As in the set-up of Proposition~\ref{uniexpint} (iv) and (v), we have the following.
	\begin{enumerate}[label=(\roman*)]
		\item The family $\{\varrho_N:N\in\NN\}$ is tight on $\cC^{\alpha-\frac d2-\eps}$.
		\item For any $\delta>0$, the family $\{\vartheta_{N,\delta}:N\in\NN\}$ is tight on $\cC^{\alpha-\frac d2-\eps}$.
	\end{enumerate}
\end{prop}
\begin{proof}
	We first prove (i). First we show that $\inf_N\cZ_N>0$. Using a Bou\'e-Dupuis approach, it suffices to use the embeddings
	\begin{align*}
		&\Big|\int_{\TT^d}\wick{Y_N^2}\, \rd x\Big| \lesssim \norm{\,\wick{Y_N^2}\,}{\cC^{2\alpha-d-\eps}} \\
		&\Big|\int_{\TT^d} Y_N\Theta_N\, \rd x\Big| \lesssim \norm{Y_N}{\cC^{\alpha-\frac d2-\eps}}\norm{\Theta}{H^{-\alpha+\frac d2+\eps}} \\
		&\hskip65pt\lesssim 1+\norm{Y_N}{\cC^{\alpha-\frac d2-\eps}}^2+\norm{\Ups_N}{H^\alpha}^2+\norm{\fZ_N}{\cC^{4\alpha-d-\eps}}^2 \\
		&\Big|\int_{\TT^d} \Theta_N^2\, \rd x\Big| \lesssim \norm{\Ups_N}{L^2}^2+\norm{\fZ_N}{\cC^{4\alpha-d-\eps}}^2
	\end{align*}
	to obtain
	\begin{align*}
		-&\log\cZ_N \\
		& \lesssim \inf_{\dot\Ups^N\in\HH_{\rm a}^\alpha}\EE\Bigl[1+\norm{\,\wick{Y_N^2}\,}{\cC^{2\alpha-d-\eps}}^\gamma+\norm{Y_N}{\cC^{\alpha-\frac d2-\eps}}^{2\gamma}+\norm{\Ups_N}{H^\alpha}^{2\gamma}+\norm{\Ups_N}{L^2}^{2\gamma}+\norm{\fZ_N}{\cC^{4\alpha-d-\eps}}^{2\gamma}\Big]
	\end{align*}
	which is enough after picking, for example, $\dot\Ups^N=0$ in the infimum. We proceed. For $\eps>0$, let $B_R\subseteq \cC^{\alpha-\frac d2-\frac\eps 2}$ be the closed ball of radius $R$ centred at the origin. The embedding $\cC^{\alpha-\frac d2-\eps}\Subset \cC^{\alpha-\frac d2-\frac\eps2}$ is compact, so $B_R$ is a compact subset of $\cC^{\alpha-\frac d2-\eps}$. We will show that, given any $\delta>0$, there exists some $R$ such that
	\begin{equation*}
		\sup_{N\in\NN} \varrho_N(B_R^c)<\delta
	\end{equation*}
	Given $M\gg 1$, let $\psi:[0,\infty)\to[0,M]$ be smooth and decreasing such that
	\begin{equation*}
		\psi(t) =
		\begin{cases} 
			M, & \text{if}\ t\le \frac R2, \\ 
			0, & \text{if}\ t>R, 
		\end{cases}
	\end{equation*}
	and define $F:\cD'\to[0,M]$ by $F(u)=\psi(\norm{u}{H^{\alpha-\frac d2-\eps}})$. Since $\inf_N\cZ_N>0$, we have
	\begin{align*}
		\rho_N(B_R^c) &\le \cZ_N^{-1}\int_{\cD'} \exp(-F(u)-V_N(u))\, \mu(\rd u) \\
		&\lesssim \int_{\cD'} \exp(-F(u_N)-V_N(u))\, \mu(\rd u).
	\end{align*}
	By the Bou\'e-Dupuis formula,
	\begin{align*}
		-\log\int_{\cD'}&\exp(-F(u_N)-V_N(u))\, \mu(\rd u) \\
		&=\inf_{\dot\Ups^N\in\HH_{\rm a}^\alpha} \EE\Bigl[F(Y_N+\Theta_N)-\sigma\int_{\TT^d}Y_N\Theta_N^2\, \rd x-\frac\sigma3\int \Theta_N^3\, \rd x \\
		&\hskip100pt+A\Big|\int_{\TT^d}\wick{(Y_N+\Theta_N)^2}\, \rd x\Big|^\gamma+\frac12\int_0^1\norm{\dot\Ups^N(t)}{H^\alpha}^2\, \rd t \\
		&\hskip100pt+\Bigl(\beta_N-\int_0^1\norm{\dot\fZ_N(t)}{H^\alpha}^2\, \rd t\Bigr)\Big]
	\end{align*}
	Now, using Lemma~\ref{pathwreg}, we have
	\begin{align*}
		\PP&\Bigl(\norm{Y_N+\Ups_N+\sigma\fZ_N}{H^{\alpha-\frac d2-\eps}}>\frac R2\Bigr) \\
		&\le \PP\Bigl(\norm{Y_N+\sigma\fZ_N}{H^{\alpha-\frac d2-\eps}}>\frac R4\Bigr)+\PP\Bigl(\norm{\Ups_N}{H^{\alpha-\frac d2-\eps}}>\frac R4\Bigr) \\
		&\le \frac12+\frac{16C}{R^2}\EE\norm{\Ups_N}{H^\alpha}^2
	\end{align*}
	where we obtained the second line by taking $R$ large enough to bound the first probability, and by using Chebyshev's inequality to bound the second. In particular,
	\begin{align*}
		\EE F(Y_N+\Ups_N+\sigma\fZ_N) &\ge \frac M2-\frac{16CM}{R^2}\EE\norm{\Ups_N}{H^\alpha}^2 \\
		&\ge \frac M2-\frac14\EE\norm{\Ups_N}{H^\alpha}^2
	\end{align*}
	after choosing $M=\frac {R^2}{64C}$ above. Arguing as follows \eqref{ZNb-dbd1} or \eqref{ZNb-dbd2} where necessary with the above and $R\gg1$, we have, uniformly in $N$,
	\begin{equation*}
		-\log\int_{\cD'}\exp(-F(u_N)-V_N(u))\, \mu(\rd u)\ge \frac M4,
	\end{equation*}
	from which the desired conclusion follows. For (ii), a similar argument using also the embeddings $W^{\alpha-\frac d2-\eps,\infty}\hookrightarrow\cA$ and $\cC^{4\alpha-d-\eps}\hookrightarrow\cA$ yields $\inf_N\cZ_{N,\delta}>0$. Arguing as before and following the proof of Proposition~\ref{uniexpint} (iii) furnishes the rest of the argument.
\end{proof}

By Propositions \ref{uniexpint} and \ref{tightness} and Prokhorov's theorem, any subsequence of $(\varrho_N)$ or $(\vartheta_{N,\delta})$ has a convergent further subsequence. By proving that subsequential limits are unique, we establish that the overall sequences $(\varrho_N)$ and $(\vartheta_{N,\delta})$ have weak limits. This is done below.

\begin{prop}[Uniqueness of weak limits] \label{uniqueness}
	As in the set-up of Proposition~\ref{uniexpint} (iv) and (v), we have the following.
	\begin{enumerate}[label=(\roman*)]
		\item Suppose that subsequences $(\varrho_{N_k^1})_{k\in\NN}$ and $(\varrho_{N_k^2})_{k\in\NN}$ of $(\varrho_N)_{N\in\NN}$ converge weakly (as measures on $\cC^{\alpha-\frac d2-\eps}$) to $\varrho^1$ and $\varrho^2$, respectively. Then $\varrho^1=\varrho^2$.
		\item There exists a choice of $s$ such that, for any $\delta>0$, the following is true. Suppose that subsequences $(\vartheta_{N_k^1,\delta})_{k\in\NN}$ and $(\vartheta_{N_k^2,\delta})_{k\in\NN}$ of $(\vartheta_{N,\delta})_{N\in\NN}$ converge weakly (as measures on $\cC^{\alpha-\frac d2-\eps}$ to $\vartheta^1_\delta$) and $\vartheta^2_\delta$, respectively. Then $\vartheta^1_\delta=\vartheta^2_\delta$.
	\end{enumerate}
\end{prop}
\begin{proof}
	We prove only (ii), as the proof of (i) is similar and easier. As a first step, we will show that
	\begin{equation} \label{uniquenesspart1}
		\lim_{k\to\infty} \cZ_{N_k^1,\delta} \ge \lim_{k\to\infty} \cZ_{N_k^2,\delta};
	\end{equation}
	without loss of generality, this implies the above is true with equality. The desired result will follow from a slight addition to the argument. By taking a further subsequence, assume that $N_k^1\ge N_k^2$ for $k=1,2,\ldots$. Let $\dot{\und\Ups}^{N_k^2}$ (and $\und\Theta_{N_k^2}=\und\Ups_{N_k^2}+\sigma\fZ_{N_k^2}$) be an $\eps$-almost optimiser for the Bou\'e-Dupuis minimisation problem in the sense that
	\begin{equation} \label{undUpsdef} \begin{aligned}
		-\log\cZ_{N_k^2,\delta} &\ge \EE\Bigl[\delta\norm{Y_{N_k^2}+\und\Ups_{N_k^2}+\sigma\fZ_{N_k^2}}{\cA}^q-\sigma\int_{\TT^d}Y_{N_k^2}\und\Theta_{N_k^2}^2\, \rd x-\frac\sigma3\int_{\TT^d}\und\Theta_{N_k^2}^3\, \rd x \\
		&\hskip 40pt+A\Big|\int_{\TT^d}\wick{(Y_{N_k^2}+\und\Theta_{N_k^2})^2}\, \rd x\Big|^\gamma+\frac12\int_0^1\norm{\dot{\und\Ups}^{N_k^2}(t)}{H^\alpha}^2\, \rd t \Big]-\eps
	\end{aligned} \end{equation}
	We now use the Bou\'e-Dupuis formula with $-\log\cZ_{N_k^1,\delta}$, and choose $\dot\Ups^{N_k^1}=\dot{\und\Ups}_{N_k^2}$ in the minimisation problem to obtain an upper bound; since $\pi_{N_k^1}\und\Ups_{N_k^2}=\und\Ups_{N_k^2}$, this reads
	\begin{equation} \label{diffbd} \begin{aligned}
		-\log&\cZ_{N_k^1,\delta}+\log\cZ_{N_k^2,\delta} \\
		&\le \delta\EE[\norm{Y_{N_k^1}+\und\Ups_{N_k^2}+\sigma\fZ_{N_k^1}}{\cA}^q-\norm{Y_{N_k^2}+\und\Ups_{N_k^2}+\sigma\fZ_{N_k^2}}{\cA}^q] \\
		&\hskip 20pt +\EE\Bigl[-\sigma\int_{\TT^d}Y_{N_k^1}(\und\Ups_{N_k^2}+\sigma\fZ_{N_k^1})^2\, \rd x-\frac\sigma3\int_{\TT^d}(\und\Ups_{N_k^2}+\sigma\fZ_{N_k^1})^3\, \rd x \\
		&\hskip 40pt+A\Big|\int_{\TT^d}\wick{(Y_{N_k^1}+\und\Ups_{N_k^2}+\sigma\fZ_{N_k^1})^2}\, \rd x\Big|^\gamma \\ 
		&\hskip 20pt \phantom{+\EE\Bigl[}\,+\sigma\int_{\TT^d}Y_{N_k^2}(\und\Ups_{N_k^2}+\sigma\fZ_{N_k^2})^2\, \rd x+\frac\sigma3\int_{\TT^d}(\und\Ups_{N_k^2}+\sigma\fZ_{N_k^2})^3\, \rd x \\
		&\hskip 40pt-A\Big|\int_{\TT^d}\wick{(Y_{N_k^2}+\und\Ups_{N_k^2}+\sigma\fZ_{N_k^2})^2}\, \rd x\Big|^\gamma\Big]+\eps.
	\end{aligned} \end{equation}
	We first prove that the first expectation appearing aboved tends to $0$ as $k\to\infty$. Using Young's inequality after factoring, there exists some constant $C>0$ so that this expectation is bounded by
	\begin{align*}
		\EE[C(\norm{Y_{N_k^1}+\und\Ups_{N_k^2}+&\sigma\fZ_{N_k^1}}{\cA}-\norm{Y_{N_k^2}+\und\Ups_{N_k^2}+\sigma\fZ_{N_k^2}}{\cA}) \\ 
		&\cdot (\delta\norm{Y_{N_k^1}+\und\Ups_{N_k^2}+\sigma\fZ_{N_k^1}}{\cA}^{q-1}+\delta\norm{Y_{N_k^2}+\und\Ups_{N_k^2}+\sigma\fZ_{N_k^2}}{\cA}^{q-1})].
	\end{align*}
	Next, using the reverse triangle inequality with the first factor and H\"older's inequality in the probability space, we obtain the successive bounds (possibly relabelling $C$ several times)
	\begin{align*}
		\EE&[C(\norm{Y_{N_k^1}-Y_{N_k^2}}{\cA}-|\sigma|\norm{\fZ_{N_k^1}-\fZ_{N_k^2}}{\cA}) \\ 
		&\hskip 60pt \cdot (\delta\norm{Y_{N_k^1}+\und\Ups_{N_k^2}+\sigma\fZ_{N_k^1}}{\cA}^{q-1}+\delta\norm{Y_{N_k^2}+\und\Ups_{N_k^2}+\sigma\fZ_{N_k^2}}{\cA}^{q-1})] \\
		&\le (\EE[C(\norm{Y_{N_k^1}-Y_{N_k^2}}{\cA}^q-|\sigma|^q\norm{\fZ_{N_k^1}-\fZ_{N_k^2}}{\cA}^q])^\frac1q \\ 
		&\hskip 60pt \cdot \Bigl(\EE\Bigl[\frac\delta2\norm{Y_{N_k^1}+\und\Ups_{N_k^2}+\sigma\fZ_{N_k^1}}{\cA}^q\Big]+\EE\Bigl[\frac\delta2\norm{Y_{N_k^2}+\und\Ups_{N_k^2}+\sigma\fZ_{N_k^2}}{\cA}^q\Bigr)\Big]^\frac{q-1}q,
	\end{align*}
	where we used Young's inequality in the second line, shifting all large constants onto $C$. As shown in the proof of Proposition~\ref{uniexpint} via a Schauder estimate and various embeddings, the first factor above decreases to $0$ as $k\to\infty$. To handle the first expectation in \eqref{diffbd}, it is enough, then, to show that the second factor above is bounded uniformly in $k$. To this end, note that
	\begin{equation*}
		\EE\Bigl[\frac\delta2\norm{Y_{N_k^2}+\und\Ups_{N_k^2}+\sigma\fZ_{N_k^2}}{\cA}^q\Big] \le -\log\cZ_{N_k^2,\delta}+\eps
	\end{equation*}
	using the definition \eqref{undUpsdef} of $\und\Ups_{N_k^2}$, and that (relabelling $C$ as necessary)
	\begin{align*}
		\EE\Bigl[\frac\delta2\norm{Y_{N_k^1}+\und\Ups_{N_k^2}+\sigma\fZ_{N_k^1}}{\cA}^q\Big] &\le \EE\Bigl[\frac\delta2\norm{Y_{N_k^2}+\und\Ups_{N_k^2}+\sigma\fZ_{N_k^2}}{\cA}^q\Big] \\ 
		&\hskip50pt+C\EE[\norm{Y_{N_k^1}-Y_{N_k^2}}{\cA}^q+|\sigma|^q\norm{\fZ_{N_k^1}-\fZ_{N_k^2}}{\cA}^q];
	\end{align*}
	in particular, noting that
	\begin{align*}
		-\log\cZ_{N_k^2,\delta} &\le \EE\Bigl[\delta\norm{Y_{N_k^2}+\sigma\fZ_{N_k^2}}{\cA}^q-\sigma^3\int_{\TT^d} Y_{N_k^2}\fZ_{N_k^2}^2\, \rd x-\frac{\sigma^4}{3}\int_{\TT^d}\fZ_{N_k^2}^3\, \rd x \\ 
		&\hskip 50pt +A\Big|\int_{\TT^d}\wick{(Y_{N_k^2}+\sigma\fZ_{N_k^2})^2}\, \rd x\Big|^\gamma\Big]
	\end{align*}
	(by taking $\Ups^{N_k^2}=0$ in the Bou\'e-Dupuis infimum) is enough, since the right-hand side above is bounded above uniformly in $k$. We move on to the second expectation in \eqref{diffbd}. Let
	\begin{equation} \label{b-d+vepartdef}
		\cE_N(\dot\Ups^N) = \EE\Bigl[\frac A2\Big|\int_{\TT^d}(2Y_N\Ups_N+\Ups_N^2)\, \rd x\Big|^\gamma+\frac12\int_0^1\norm{\dot\Ups^N(t)}{H^\alpha}^2\, \rd t\Big]
	\end{equation}
	be the ``positive part'' appearing in the Bou\'e-Dupuis expansion of $-\log\cZ_N$. Then, by Lemmas \ref{pathwreg} and \ref{phi33heart}, we have
	\begin{equation} \label{b-d+vepartbd}
		\EE[\norm{\und\Ups_{N_k^2}}{H^\alpha}^2+\norm{\und\Ups_{N_k^2}}{L^2}^{2\gamma}]\lesssim 1+\cE_{N_k^2}(\dot{\und\Ups}^{N_k^2}).
	\end{equation}
	The contribution to the second expectation in \eqref{diffbd} from the terms $-\sigma\int_{\TT^d}Y_{N_k^j}\und\Theta_{N_k^j}^2\, \rd x$, $j=1,2$, can be written as
	\begin{align*}
		&-\sigma\EE\Bigl[\int_{\TT^d}(Y_{N_k^1}-Y_{N_k^2})\und\Ups_{N_k^2}^2\, \rd x\Bigr]-\sigma^2\EE\Bigl[\int_{\TT^d}(Y_{N_k^1}-Y_{N_k^2})(2\und\Ups_{N_k^2}+\sigma\fZ_{N_k^1})\fZ_{N_k^1}\, \rd x\Bigr] \\
		&\hskip 50pt-\sigma^2\EE\Bigl[\int_{\TT^d}Y_{N_k^2}(\fZ_{N_k^1}-\fZ_{N_k^2})(2\und\Ups_{N_k^2}+\sigma\fZ_{N_k^1}+\sigma\fZ_{N_k^2})\, \rd x\Bigr].
	\end{align*}
	Now, we calculate, using Lemma~\ref{besovest}, H\"older's inequality in the probability space followed by Young's inequality, and \eqref{b-d+vepartbd}, that
	\begin{align*}
		\Big|\EE&\Bigl[\int_{\TT^d}(Y_{N_k^1}-Y_{N_k^2})\und\Ups_{N_k^2}^2\, \rd x\Big]\Big| \\ 
		&\lesssim \EE[\norm{Y_{N_k^1}-Y_{N_k^2}}{\cC^{\alpha-\frac d2-\eps}}\norm{\und\Ups_{N_k^2}}{H^{-\alpha+\frac d2+2\eps}}\norm{\und\Ups_{N_k^2}}{L^2}] \\
		&\lesssim \EE[\norm{Y_{N_k^1}-Y_{N_k^2}}{\cC^{\alpha-\frac d2-\eps}}\norm{\und\Ups_{N_k^2}}{H^\alpha}^{-1+\frac{d}{2\alpha}+\frac{2\eps}{\alpha}}\norm{\und\Ups_{N_k^2}}{L^2}^{3-\frac{d}{2\alpha}-\frac{2\eps}{\alpha}}] \\
		&\lesssim \norm{Y_{N_k^1}-Y_{N_k^2}}{L^c(\PP;\cC^{\alpha-\frac d2-\eps})}(1+\EE\norm{\und\Ups_{N_k^2}}{H^\alpha}^2+\EE\norm{\und\Ups_{N_k^2}}{L^2}^{\frac{d}{d-2\alpha}}) \\
		&\lesssim \norm{Y_{N_k^1}-Y_{N_k^2}}{L^c(\PP;\cC^{\alpha-\frac d2-\eps})}(1+\cE_{N_k^2}(\dot{\und\Ups}^{N_k^2}))
	\end{align*}
	for some large exponent $c>1$, where we used $d<4\alpha$ from the third line to the fourth. Using the same techniques, we have
	\begin{align*}
		\Big|\EE&\Bigl[\int_{\TT^d}(Y_{N_k^1}-Y_{N_k^2})(2\und\Ups_{N_k^2}+\sigma\fZ_{N_k^1})\fZ_{N_k^1}\, \rd x\Big]\Big| \\
		&\lesssim \EE[\norm{Y_{N_k^1}-Y_{N_k^2}}{\cC^{\alpha-\frac d2-\eps}}\norm{(2\und\Ups_{N_k^2}+\sigma\fZ_{N_k^1})\fZ_{N_k^1}}{H^{-\alpha+\frac d2+2\eps}}] \\
		&\lesssim \EE[\norm{Y_{N_k^1}-Y_{N_k^2}}{\cC^{\alpha-\frac d2-\eps}}\norm{2\und\Ups_{N_k^2}+\sigma\fZ_{N_k^1}}{H^{-\alpha+\frac d2+2\eps}}\norm{\fZ_{N_k^1}}{H^{-\alpha+\frac d2+2\eps}}] \\
		&\lesssim \EE[\norm{Y_{N_k^1}-Y_{N_k^2}}{\cC^{\alpha-\frac d2-\eps}}\norm{\fZ_{N_k^1}}{\cC^{4\alpha-d-\eps}}(\norm{\und\Ups_{N_k^2}}{H^\alpha}+\norm{\fZ_{N_k^1}}{\cC^{4\alpha-d-\eps}})] \\
		&\lesssim \norm{Y_{N_k^1}-Y_{N_k^2}}{L^c(\PP;\cC^{\alpha-\frac d2-\eps})}(1+\cE_{N_k^2}(\dot{\und\Ups}^{N_k^2}))
	\end{align*}
	for some large exponent $c>1$, possibly relabelled. Here we required $d<\frac{10}3\alpha$. Onwards,
	\begin{align*}
		\Big|&\EE\Bigl[\int_{\TT^d} Y_{N_k^2}(\fZ_{N_k^1}-\fZ_{N_k^2})(2\und\Ups_{N_k^2}+\sigma\fZ_{N_k^1}+\sigma\fZ_{N_k^2})\, \rd x\Big]\Big| \\
		&\lesssim \EE[\norm{Y_{N_k^2}(\fZ_{N_k^1}-\fZ_{N_k^2})}{H^{\alpha-\frac d2-2\eps}}\norm{2\und\Ups_{N_k^2}+\sigma\fZ_{N_k^1}+\sigma\fZ_{N_k^2}}{H^{-\alpha+\frac d2+2\eps}}] \\
		&\lesssim \EE[\norm{\fZ_{N_k^1}-\fZ_{N_k^2}}{\cC^{4\alpha-d-\eps}}\norm{Y_{N_k^2}}{\cC^{\alpha-\frac d2-\eps}}(\norm{\und\Ups_{N_k^2}}{H^\alpha}+\norm{\fZ_{N_k^1}}{\cC^{4\alpha-d-\eps}}+\norm{\fZ_{N_k^2}}{\cC^{4\alpha-d-\eps}})] \\
		&\lesssim \norm{\fZ_{N_k^1}-\fZ_{N_k^2}}{L^c(\PP;\cC^{4\alpha-d-\eps})}^c(1+\cE_{N_k^2}(\dot{\und\Ups}^{N_k^2}))
	\end{align*}
	for some $c>1$, possibly relabelled. Next, we will express the contribution to the second expectation in \eqref{diffbd} from the terms $-\sigma\int_{\TT^d} \und\Theta_{N_k^j}^3\, \rd x$, $j=1,2$, as
	\begin{equation} \begin{aligned}
		-\sigma^2\int_{\TT^d}\und\Ups_{N_k^2}^2&(\fZ_{N_k^1}-\fZ_{N_k^2})\, \rd x-\sigma^3\int_{\TT^d}\und\Ups_{N_k^2}(\fZ_{N_k^1}+\fZ_{N_k^2})(\fZ_{N_k^1}-\fZ_{N_k^2})\, \rd x \\
		&-\frac{\sigma^4}{3}\int_{\TT^d}(\fZ_{N_k^1}^2+\fZ_{N_k^1}\fZ_{N_k^2}+\fZ_{N_k^2}^2)(\fZ_{N_k^1}-\fZ_{N_k^2})\, \rd x.
	\end{aligned} \end{equation}
	Hence we now work to bound the above (under expectation). Proceeding as before, we have
	\begin{align*}
		\Big|\EE\Bigl[\int_{\TT^d}\und\Ups_{N_k^2}^2(\fZ_{N_k^1}-\fZ_{N_k^2})\, \rd x\Big]\Big| &\lesssim \EE[\norm{\fZ_{N_k^1}-\fZ_{N_k^2}}{\cC^{4\alpha-d-\eps}}\norm{\und\Ups_{N_k^2}^2}{H^\alpha}] \\
		&\lesssim \EE[\norm{\fZ_{N_k^1}-\fZ_{N_k^2}}{\cC^{4\alpha-d-\eps}}\norm{\und\Ups_{N_k}}{H^\alpha}\norm{\und\Ups_{N_k^2}}{L^2}] \\
		&\lesssim \norm{\fZ_{N_k^1}-\fZ_{N_k^2}}{L^c(\PP;\cC^{4\alpha-d-\eps})}(1+\cE_{N_k^2}(\dot{\und\Ups}^{N_k^2}))
	\end{align*}
	for some $c>1$. Next,
	\begin{align*}
		\Big|\EE&\Bigl[\int_{\TT^d}\und\Ups_{N_k^2}(\fZ_{N_k^1}+\fZ_{N_k^2})(\fZ_{N_k^1}-\fZ_{N_k^2})\, \rd x\Big]\Big| \\
		&\lesssim \EE[\norm{(\fZ_{N_k^1}-\fZ_{N_k^2})(\fZ_{N_k^1}+\fZ_{N_k^2})}{H^{4\alpha-d-2\eps}}\norm{\und\Ups_{N_k^2}}{H^{-4\alpha+d+2\eps}}] \\
		&\lesssim \EE[\norm{\fZ_{N_k^1}-\fZ_{N_k^2}}{\cC^{4\alpha-d-\eps}}(\norm{\fZ_{N_k^1}}{\cC^{4\alpha-d-\eps}}+\norm{\fZ_{N_k^2}}{\cC^{4\alpha-d-\eps}})\norm{\und\Ups_{N_k^2}}{H^\alpha}] \\
		&\lesssim \norm{\fZ_{N_k^1}-\fZ_{N_k^2}}{L^c(\PP;\cC^{4\alpha-d-\eps})}(1+\cE_{N_k^2}(\dot{\und\Ups}^{N_k^2}))
	\end{align*}
	for some $c>1$. Finally,
	\begin{align*}
		\Big|\EE&\Bigl[\int_{\TT^d}(\fZ_{N_k^1}^2+\fZ_{N_k^1}\fZ_{N_k^2}+\fZ_{N_k^2}^2)(\fZ_{N_k^1}-\fZ_{N_k^2})\, \rd x\Big]\Big| \\
		&\lesssim \EE[\norm{\fZ_{N_k^1}-\fZ_{N_k^2}}{\cC^{4\alpha-d-\eps}}(\norm{\fZ_{N_k^1}}{\cC^{4\alpha-d-\eps}}^2+\norm{\fZ_{N_k^2}}{\cC^{4\alpha-d-\eps}}^2)] \\
		&\lesssim \norm{\fZ_{N_k^1}-\fZ_{N_k^2}}{L^2(\PP;\cC^{4\alpha-d-\eps})}.
	\end{align*}
	We treat the contribution to the second expectation in \eqref{diffbd} from the terms $A|\int_{\TT^d}\wick{(Y_{N_k^j}+\und\Ups_{N_k^2}+\sigma\fZ_{N_k^j})^2}\, \rd x|^\gamma$, where $j=1,2$. Here, by factoring and using the reverse triangle, Young's, and H\"older's inequalities, we find
	\begin{equation} \label{tamediffbd} \begin{aligned}
		\EE&\Bigl[\Big|\int_{\TT^d}(\wick{Y_{N_k^1}^2}+2Y_{N_k^1}(\und\Ups_{N_k^2}+\sigma\fZ_{N_k^1})+(\und\Ups_{N_k^2}+\sigma\fZ_{N_k^1})^2)\, \rd x\Big|^\gamma \\
		&\hskip30pt-\Big|\int_{\TT^d}(\wick{Y_{N_k^2}^2}+2Y_{N_k^2}(\und\Ups_{N_k^2}+\sigma\fZ_{N_k^2})+(\und\Ups_{N_k^2}+\sigma\fZ_{N_k^2})^2)\, \rd x\Big|^\gamma\Big] \\
		&\lesssim \Bigl(\bignorm{\int_{\TT^d}(\wick{Y_{N_k^1}^2}-\wick{Y_{N_k^2}^2})\, \rd x}{L^\gamma(\PP)}+\bignorm{\int_{\TT^d}(Y_{N_k^1}-Y_{N_k^2})\und\Ups_{N_k^2}\, \rd x}{L^\gamma(\PP)} \\
		&\hskip50pt+\bignorm{\int_{\TT^d}(Y_{N_k^1}-Y_{N_k^2})\fZ_{N_k^1}\, \rd x}{L^\gamma(\PP)}+\bignorm{\int_{\TT^d}Y_{N_k^2}(\fZ_{N_k^1}-\fZ_{N_k^2})\, \rd x}{L^\gamma(\PP)} \\
		&\hskip50pt+\bignorm{\int_{\TT^d}(\fZ_{N_k^1}-\fZ_{N_k^2})(2\und\Ups_{N_k^2}+\sigma\fZ_{N_k^1}+\sigma\fZ_{N_k^2})\, \rd x}{L^\gamma(\PP)}\Bigr) \\
		&\hskip20pt \cdot \Bigl(\bignorm{\int_{\TT^d}\wick{(Y_{N_k^1}+\und\Ups_{N_k^2}+\sigma\fZ_{N_k^1})^2}\, \rd x}{L^\gamma(\PP)}^{\gamma-1} \\
		&\hskip50pt+\bignorm{\int_{\TT^d}\wick{(Y_{N_k^2}+\und\Ups_{N_k^2}+\sigma\fZ_{N_k^2})^2}\, \rd x}{L^\gamma(\PP)}^{\gamma-1}\Bigr);
	\end{aligned} \end{equation}
	since $N_k^1\ge N_k^2$, we can use Plancherel's theorem and observations on disjoint Fourier supports to obtain
	\begin{equation*}
		\int_{\TT^d}(Y_{N_k^1}-Y_{N_k^2})\und\Ups_{N_k^2}\, \rd x = \int_{\TT^d}Y_{N_k^2}(\fZ_{N_k^1}-\fZ_{N_k^2})\, \rd x = \int_{\TT^d}(\fZ_{N_k^1}-\fZ_{N_k^2})\und\Ups_{N_k^2}\, \rd x=0;
	\end{equation*}
	then, using various embeddings (Lemma~\ref{besovest}) the first factor on the right-hand side of \eqref{tamediffbd} is bounded, up to a multiplicative constant, by
	\begin{align*}
		\norm{\,\wick{Y_{N_k^1}^2}-&\wick{Y_{N_k^2}^2}\,}{L^\gamma(\PP;\cC^{2\alpha-d-\eps})}+\norm{Y_{N_k^1}-Y_{N_k^2}}{L^{2\gamma}(\PP;\cC^{\alpha-\frac d2-\eps})}\norm{\fZ_{N_k^1}}{L^{2\gamma}(\PP;\cC^{4\alpha-d-\eps})} \\
		&+\norm{\fZ_{N_k^1}-\fZ_{N_k^2}}{L^{2\gamma}(\PP;\cC^{4\alpha-d-\eps})}(\norm{\fZ_{N_k^1}}{L^{2\gamma}(\PP;\cC^{4\alpha-d-\eps})}+\norm{\fZ_{N_k^2}}{L^{2\gamma}(\PP;\cC^{4\alpha-d-\eps})}),
	\end{align*}
	and tends to $0$ as $N_k^1,N_k^2\to\infty$. We now establish a uniform upper bound on the second factor on the right-hand side of \eqref{tamediffbd}. The integral in the second term in this factor can be written as
	\begin{align*}
		\int_{\TT^d}\wick{(Y_{N_k^2}+\und\Ups_{N_k^2}&+\sigma\fZ_{N_k^2})^2}\, \rd x+\int_{\TT^d}(\wick{Y_{N_k^1}^2}-\wick{Y_{N_k^2}^2})\, \rd x+2\sigma\int_{\TT^d}Y_{N_k^1}(\fZ_{N_k^1}-\fZ_{N_k^2})\, \rd x \\
		&+2\sigma\int_{\TT^d} \fZ_{N_k^2}(Y_{N_k^1}-Y_{N_k^2})\, \rd x+\sigma^2\int_{\TT^d}(\fZ_{N_k^1}-\fZ_{N_k^2})(\fZ_{N_k^1}+\fZ_{N_k^2})\, \rd x
	\end{align*}
	and so it suffices to find a uniform bound on the first. First suppose that $\gamma=\frac{d}{d-2\alpha}$. Then, by Lemma~\ref{phi33heart} and the subsequent remark, we have
	\begin{align*}
		\EE\Bigl[\Big|\int_{\TT^d}\wick{(Y_{N_k^2}+\und\Theta_{N_k^2})^2}\, \rd x\Big|^\frac{d}{d-2\alpha}\Big] &\lesssim 1+\EE\Bigl[\Big|\int_{\TT^d}Y_{N_k^2}\und\Ups_{N_k^2}\Big|^\frac{d}{d-2\alpha}\Big]+\EE\norm{\und\Ups_{N_k^2}}{L^2}^\frac{2d}{d-2\alpha} \\
		&\lesssim 1+\cE_{N_k^2}(\und\Ups_{N_k^2}).
	\end{align*}
	Now, suppose $\gamma>\frac{d}{d-2\alpha}$. Proceeding as in the proof of Proposition~\ref{uniexpint} (v), we have
	\begin{align*}
		\EE\Bigl[-\delta\norm{Y_{N_k^2}+\und\Theta_{N_k^2}}{\cA}^q+&\sigma\int_{\TT^d}Y_{N_k^2}\und\Theta_{N_k^2}^2\, \rd x+\frac\sigma3\int_{\TT^d}\und\Theta_{N_k^2}^3\, \rd x\Big] \\
		&\le \EE\Bigl[A'\Big|\int_{\TT^d}\wick{(Y_{N_k^2}+\und\Theta_{N_k^2})^2}\, \rd x\Big|^\frac{d}{d-2\alpha}+\frac14\norm{\und\Ups_{N_k^2}}{H^\alpha}^2\Big]+C
	\end{align*}
	for some new $A'$ and large constant $C$. By virtue of \eqref{undUpsdef}, our choice of $\und\Ups_{N_k^2}$ as an $\eps$-optimiser for $-\log\cZ_{N_k^2,\delta}$, it therefore follows that
	\begin{equation*}
		\EE\Bigl[A'\Big|\int_{\TT^d}\wick{(Y_{N_k^2}+\und\Theta_{N_k^2})^2}\, \rd x\Big|^\frac{d}{d-2\alpha}-A\Big|\int_{\TT^d}\wick{(Y_{N_k^2}+\und\Theta_{N_k^2})^2}\, \rd x\Big|^\gamma\Big]\ge\log\cZ_{N_k^2,\delta}+C,
	\end{equation*}
	where $C$ is possibly re-labelled. However, when $\gamma>\frac{d}{d-2\alpha}$, one has $A'r^{\frac{d}{d-2\alpha}}-Ar^\gamma\le-\frac A2r^\gamma+C$ for any $r>0$, for some large $C$, and so we conclude that
	\begin{equation*}
		\EE\Bigl[\Big|\int_{\TT^d}\wick{(Y_{N_k^2}+\und\Theta_{N_k^2})^2}\, \rd x\Big|^\gamma\Big] \le -\frac2A\log\cZ_{N_k^2,\delta}+C,
	\end{equation*}
	and the above is bounded uniformly in $N_k^2$. It remains to bound $\cE_{N_k^2}(\dot{\und\Ups}^{N_k^2})$. Arguing as in the beginning of the proof of Proposition~\ref{uniexpint} (v), there exist $\delta'>0$ and some constant $C(\sigma,\delta')$ such that
	\begin{align*}
		-\log\cZ_{N,\delta} &\ge \inf_{\dot\Ups^N\in\HH_{\rm a}^\alpha}\Bigl[\EE\Bigl[\norm{Y_N+\Ups_N+\sigma\fZ_N}{\cA}^q-\frac\sigma3\int_{\TT^d}(\Ups_N+\sigma\fZ_N)\, \rd x+\frac{\delta'}{2}\norm{\Ups}{L^2}^{2\gamma} \\
		&\hskip100pt +\frac18\norm{\Ups}{H^\alpha}^2\Big]+\frac14\cE_N(\dot\Ups^N)\Big]-C(\sigma,\delta').
	\end{align*}	
	Proceeding as in the remainder of the proof to bound the first term in the infimum and using \eqref{undUpsdef}, we have that
	\begin{equation*}
		\cE_{N_k^2}(\dot{\und\Ups}^{N_k^2})\lesssim -\log\cZ_{N_k^2,\delta}+C
	\end{equation*}
	for some $C>0$. But the right-hand side of the above is bounded above uniformly in $k\in\NN$, from which we obtain the desired \eqref{uniquenesspart1}. Next, we prove that $\vartheta_\delta^1=\vartheta_\delta^2$. As done previously, it suffices to establish that for any bounded Lipschitz $F:\cC^{\alpha-\frac d2-\eps}\to \RR$, assuming $N_k^1\ge N_k^2$, we have
	\begin{equation}
		\lim_{k\to\infty}\int\exp(F(u))\, \vartheta_{N_k^1,\delta}(\rd u)\ge \lim_{k\to\infty}\int\exp(F(u))\, \vartheta_{N_k^2,\delta}(\rd u).
	\end{equation}
	In fact, as above, and since $F$ is bounded, it suffices to show
	\begin{align*}
		\limsup_{k\to\infty}\Bigl[-\log\Bigl(&\int \exp(F(u_{N_k^1})-W_{N_k^1,\delta}(u))\, \mu(\rd u)\Bigr) \\
		&+\log\Bigl(\int \exp(F(u_{N_k^2})-W_{N_k^2,\delta}(u))\, \mu(\rd u)\Bigr)\Big]\le 0.
	\end{align*}
	By picking an $\eps$-optimiser for the Bou\'e-Dupuis minimisation problem as done previously, the left-hand side of the above is bounded by
	\begin{align*}
		&\EE\Bigl[-F(Y_{N_k^1}+\und\Ups_{N_k^2}+\sigma\fZ_{N_k^1})+\delta\norm{Y_{N_k^1}+\und\Ups_{N_k^2}+\sigma\fZ_{N_k^1}}{\cA}^q \\
		&\hskip60pt-\sigma\int_{\TT^d} Y_{N_k^1}(\und\Ups_{N_k^2}+\sigma\fZ_{N_k^1})^2\, \rd x-\frac\sigma3\int_{\TT^d} (\und\Ups_{N_k^2}+\sigma\fZ_{N_k^1})^3\, \rd x \\
		&\hskip60pt +A\Big|\wick{(Y_{N_k^1}+\und\Ups_{N_k^2}+\fZ_{N_k^1})^2}\, \rd x\Big|^\gamma+\frac12\int_0^1\norm{\dot{\und\Ups}_{N_k^2}(t)}{H^\alpha}^2\, \rd x\Big] \\
		&\hskip10pt+\EE\Bigl[F(Y_{N_k^2}+\und\Ups_{N_k^2}+\sigma\fZ_{N_k^2})-\delta\norm{Y_{N_k^2}+\und\Ups_{N_k^2}+\sigma\fZ_{N_k^2}}{\cA}^q \\
		&\hskip70pt +\sigma\int_{\TT^d} Y_{N_k^2}(\und\Ups_{N_k^2}+\sigma\fZ_{N_k^2})^2\, \rd x+\frac\sigma3\int_{\TT^d} (\und\Ups_{N_k^2}+\sigma\fZ_{N_k^2})^3\, \rd x \\
		&\hskip70pt -A\Big|\int_{\TT^d}\wick{(Y_{N_k^2}+\und\Ups_{N_k^2}+\fZ_{N_k^2})^2}\, \rd x\Big|^\gamma-\frac12\int_0^1\norm{\dot{\und\Ups}^{N_k^2}(t)}{H^\alpha}^2\, \rd x\Big]+\eps.
	\end{align*}
	Given \eqref{uniquenesspart1}, it suffices to prove that
	\begin{equation*}
		\lim_{k\to\infty}\EE|-F(Y_{N_k^1}+\und\Ups_{N_k^2}+\sigma\fZ_{N_k^1})+F(Y_{N_k^2}+\und\Ups_{N_k^2}+\sigma\fZ_{N_k^2})|=0.
	\end{equation*}
	Say $F$ is $C$-Lipschitz, so that the expectation under the limit is bounded by
	\begin{equation*}
		C\EE\norm{(Y_{N_k^1}-Y_{N_k^2})+\sigma(\fZ_{N_k^1}-\fZ_{N_k^2})}{\cC^{\alpha-\frac d2-\eps}},
	\end{equation*}
	which is enough to conclude the proof.
\end{proof}

To complete our program of construction, we require the following proposition to make sense of the $\sigma$-finite version $\varrho^\delta$ of $\Phi^3_d$ in the strongly nonlinear case.

\begin{prop}[$\vartheta_\delta$-a.s. finiteness of the $\cA$ norm] \label{asfiniteness}
	One has $\norm{u}{\cA}<\infty$ for $\vartheta_\delta$-a.e. $u$, and, in particular, the measure
	\begin{equation}
		\varrho^\delta(\rd u) = \exp(\delta\norm{u}{\cA}^q)\, \vartheta_\delta(\rd u)
	\end{equation}
	is well-defined.
\end{prop}
\begin{proof}
	Let $\hat\varphi_1\in C_c^\infty(\RR^d)$ be radial with $\norm{\hat\varphi_1}{L^2(\RR^d)}=1$, and set 
	\begin{equation*}
		\hat\varphi(\xi) = \int_{\RR^d} \hat\varphi_1(\xi-\eta)\hat{\bar\varphi}_1(-\eta)\, \rd\eta.
	\end{equation*}
	For $\eps>0$, define the periodic function $\varphi_\eps$ by its Fourier coefficients $\hat\varphi_\eps(n)=\hat\varphi(\eps n)$. As $\hat\varphi$ has compact support, there exists $N_0$ depending on $\eps$ such that $\varphi_\eps\ast u=\varphi_\eps\ast u_N$ for $N\ge N_0$. By the Poisson summation formula,
	\begin{equation*}
		\varphi_\eps(x)=\sum_{m\in\ZZ^d}\eps^{-d}|\cF^{-1}_{\RR^d}\hat\varphi_1(\eps^{-1}(x+m))|^2
	\end{equation*}
	where $\cF_{\RR^d}$ is the Fourier transform on $\RR^d$. That is, $\varphi_\eps\ge 0$. Moreover, $\norm{\varphi_\eps}{L^1}$ is nothing but $\varphi(0)$, which is just $\norm{\varphi_1}{L^2(\RR^d)}^2=1$. Hence by Young's convolution inequality,
	\begin{equation*}
		\norm{\varphi_\eps\ast u}{\cA}\le\norm{u}{\cA}.
	\end{equation*}
	Finally, $(\varphi_\eps)_{\eps>0}$ is an approximation to the identity, and so $\varphi_\eps\ast u\to u$ in $\cA$ as $\eps\downarrow 0$. 

	Next, let $\chi:[0,\infty)\to[0,1]$ be smooth and decreasing, such that $\chi=1$ on $[0,1]$ and $\chi=0$ on $(2,\infty)$. By the embedding $\cC^{\alpha-\frac d2-\eps}\hookrightarrow\cA$, for any $M>0$ and any $u\in\cC^{\alpha-\frac d2-\eps}$, we have
	\begin{equation*}
		\norm{u}{\cA}\chi\Bigg(\frac{\norm{u}{\cC^{\alpha-\frac d2-\eps}}}{M}\Bigg) \lesssim M.
	\end{equation*}
	Hence, by monotone convergence, Fatou's lemma with $(\varphi_\eps)$ acting as an approximate identity, properties of $\varphi_\eps$ discussed above, and the weak convergence $\vartheta_{N,\delta}\overset{\ast}{\rightharpoonup}\vartheta_\delta$, we have
	\begin{equation*}
		\int\norm{u}{\cA}\, \vartheta_\delta(\rd u) \le \lim_{M\to\infty}\liminf_{\eps\to 0}\lim_{N\to\infty}\int \norm{\varphi_\eps\ast u_N}{\cA}\chi\Bigg(\frac{\norm{u}{\cC^{\alpha-\frac d2-\eps}}}{M}\Bigg)\, \vartheta_{N,\delta}(\rd u)
	\end{equation*}
	Using the bound $\norm{\varphi_\eps\ast u}{\cA}\le\norm{u}{\cA}$, the fact that $\chi\le 1$, and the definition of $\vartheta_{N,\delta}$, the above is bounded by (a constant multiple of)
	\begin{equation*}
		\lim_{N\to\infty}\int \norm{u_N}{\cA}\exp(-\delta\norm{u_N}{\cA}^q-V_N(u))\, \mu(\rd u),
	\end{equation*}
	which is a finite quantity as can be observed by $\norm{u_N}{\cA}\lesssim_{\delta,q} \exp(\frac\delta2\norm{u_N}{\cA}^q)$ and the uniform exponential integrability of Proposition~\ref{uniexpint} (ii) applied to $(\vartheta_{N,\frac\delta2})$. Hence $\norm{u}{\cA}<\infty$ for $\vartheta_\delta$-a.e. $u$, thus completing the proof.
\end{proof}

To conclude the subsection, we include here the proofs of Lemmas \ref{b-dcrossterms} and \ref{phi33heart}.

\begin{proof}[Proof of Lemma~\ref{b-dcrossterms}.] 
	For \eqref{YNTN2bd}, use, in order, \eqref{besovduality}, \eqref{immediatebesovembeddings}, \eqref{besovfractionalleibniz}, \eqref{Upsdef} and \eqref{immediatebesovembeddings}, and \eqref{sobolevinterpolation} and \eqref{immediatebesovembeddings} to write
	\begin{align*}
		\Big|\int_{\TT^d} Y_N\Theta_N^2\, \rd x\Big| &\lesssim \norm{Y_N}{H^{\alpha-\frac{d}{2}-2\eps}}\norm{\Theta_N^2}{H^{-\alpha+\frac{d}{2}+2\eps}} \\
		&\lesssim \norm{Y_N}{\cC^{\alpha-\frac{d}{2}-\eps}}\norm{\Theta_N^2}{H^{-\alpha+\frac{d}{2}+2\eps}} \\
		&\lesssim \norm{Y_N}{\cC^{\alpha-\frac{d}{2}-\eps}}\norm{\Theta_N}{H^{-\alpha+\frac{d}{2}+2\eps}}\norm{\Theta_N}{L^2} \\
		&\lesssim \norm{Y_N}{\cC^{\alpha-\frac{d}{2}-\eps}}(\norm{\Ups_N}{H^{-\alpha+\frac{d}{2}+2\eps}}(\norm{\Ups_N}{L^2}+\norm{\fZ_N}{\cC^{-\alpha+\frac{d}{2}+2\eps}}) \\ 
		&\hskip100pt+\norm{\fZ_N}{\cC^{-\alpha+\frac{d}{2}+2\eps}}^2) \\
		&\lesssim \norm{Y_N}{\cC^{\alpha-\frac{d}{2}-\eps}}(\norm{\Ups_N}{L^2}^{2-\frac{d}{2\alpha}-\frac{2\eps}{\alpha}}\norm{\Ups_N}{H^\alpha}^{-1+\frac{d}{2\alpha}+\frac{2\eps}{\alpha}}(\norm{\Ups_N}{L^2}+\norm{\fZ_N}{\cC^{4\alpha-d-\eps}}) \\ 
		&\hskip 100pt+\norm{\fZ_N}{\cC^{4\alpha-d-\eps}}^2); \\
	\end{align*} 
	now \eqref{YNTN2bd} follows after Young's inequality; we must assume $d\le 3\alpha$ for $\frac{12\alpha-2d}{4\alpha-d}\le\frac{2d}{d-2\alpha}$. From here one has \eqref{YNTN2bd}. For \eqref{TN3bd}, we proceed as follows: recall \eqref{Upsdef}, and use, in order, \eqref{besovembeddings}, \eqref{sobolevinterpolation}, and Young's inequality to write
	\begin{align*}
		\Big|\int_{\TT^d} \Ups_N^3\, \rd x\Big| &\lesssim \norm{\Ups_N}{H^{\frac{d}{6}}}^3 \\
		&\lesssim \norm{\Ups_N}{L^2}^{3-\frac{d}{2\alpha}}\norm{\Ups_N}{H^\alpha}^{\frac{d}{2\alpha}} \\
		&\lesssim C(\delta)\norm{\Ups_N}{L^2}^\frac{12\alpha-2d}{4\alpha-d}+\delta\norm{\Ups_N}{H^\alpha}^2;
	\end{align*} 
	next, use, in order, \eqref{besovduality}, \eqref{immediatebesovembeddings}, \eqref{besovfractionalleibniz}, \eqref{immediatebesovembeddings} to write
	\begin{align*}
		\Big|\int_{\TT^d} \Ups_N^2\fZ_N\, \rd x\Big| &\lesssim \norm{\Ups_N^2}{H^{-4\alpha+d+2\eps}}\norm{\fZ_N}{\cC^{4\alpha-d-\eps}} \\
		&\lesssim \norm{\Ups_N}{L^2}\norm{\Ups_N}{H^{-4\alpha+d+2\eps}}\norm{\fZ_N}{\cC^{4\alpha-d-\eps}} \\
		&\lesssim \norm{\Ups_N}{L^2}\norm{\Ups_N}{H^\alpha}\norm{\fZ_N}{\cC^{4\alpha-d-\eps}}
	\end{align*} 
	and use Young's inequality; continuing, use, in order, \eqref{besovduality}, \eqref{besovfractionalleibniz}, and \eqref{immediatebesovembeddings} to write
	\begin{align*}
		\Big|\int_{\TT^d}\Ups_N\fZ_N^2\, \rd x\Big| &\le \norm{\Ups_N}{H^\alpha}\norm{\fZ_N^2}{H^{-\alpha}} \\
		&\lesssim \norm{\Ups_N}{H^\alpha}\norm{\fZ_N}{L^2}\norm{\fZ_N}{H^{-\alpha}} \\
		&\lesssim \norm{\Ups_N}{H^\alpha}\norm{\fZ_N}{\cC^{4\alpha-d-\eps}}^2
	\end{align*}
	and use Young's inequality; finally, note simply by \eqref{immediatebesovembeddings} that
	\begin{equation*}
		\Big|\int_{\TT^d} \fZ_N^3\, \rd x\Big| \lesssim \norm{\fZ_N}{\cC^{4\alpha-d-\eps}}^3,
	\end{equation*}
	thus yielding \eqref{TN3bd}. Next, we prove \eqref{tamingtermbd}. First observe that
	\begin{equation*} \begin{aligned}
		A\Big|\int_{\TT^d}\wick{(Y_N+\Theta_N)^2}\, \rd x\Big|^\gamma &\ge \frac{A}{2}\Big|\int_{\TT^d} (2Y_N\Ups_N+\Ups_N)\, \rd x\Big|^\gamma \\ 
		&\hskip12pt-C\Big|\int_{\TT^d} (\wick{Y_N^2}+2\sigma Y_N\fZ_N+2\sigma\Ups_N\fZ_N+\sigma^2\fZ_N^2)\, \rd x\Big|^\gamma.
	\end{aligned} \end{equation*}
	Now, using, in order, \eqref{besovduality}, \eqref{immediatebesovembeddings}, and Young's inequality, we have
	\begin{align*}
		\Big|\int_{\TT^d}Y_N\fZ_N\, \rd x\Big|^\gamma &\le \norm{Y_N}{H^{\alpha-\frac{d}{2}-2\eps}}^\gamma\norm{\fZ_N}{H^{-\alpha+\frac{d}{2}+2\eps}}^\gamma \\
		&\lesssim \norm{Y_N}{\cC^{\alpha-\frac{d}{2}-\eps}}^\gamma\norm{\fZ_N}{\cC^{4\alpha-d-\eps}}^\gamma \\
		&\lesssim \norm{Y_N}{\cC^{\alpha-\frac{d}{2}-\eps}}^{2\gamma}+\norm{\fZ_N}{\cC^{4\alpha-d-\eps}}^{2\gamma}.
	\end{align*}
	Using, in order, \eqref{besovduality}, \eqref{immediatebesovembeddings}, and Young's inequality, we have
	\begin{align*}
		\Big|\int_{\TT^d}\Ups_N\fZ_N\, \rd x\Big|^\gamma &\le \norm{\Ups_N}{L^2}^\gamma\norm{\fZ_N}{L^2}^\gamma \\
		&\lesssim \norm{\Ups_N}{L^2}^\gamma\norm{\fZ_N}{\cC^{4\alpha-d-\eps}}^\gamma \\
		&\lesssim \frac{\delta}{C}\norm{\Ups_N}{L^2}^{2\gamma}+C'\norm{\fZ_N}{\cC^{4\alpha-d-\eps}}^{2\gamma}.
	\end{align*}
	Finally, using \eqref{immediatebesovembeddings},
	\begin{align*}
		\Big|\int_{\TT^d} \fZ_N^2\, \rd x\Big|^\gamma &\le \norm{\fZ_N}{L^2}^\frac\gamma2 \\
		&\lesssim \norm{\fZ_N}{\cC^{4\alpha-d-\eps}}^\frac{\gamma}{2}.
	\end{align*}
	Putting the final four displays together yields \eqref{tamingtermbd}. Finally, to prove \eqref{tamingtermbdeasy}, we assume $d = 2\alpha$ and proceed as in the proof of \eqref{tamingtermbd}, using also Young's inequality and \eqref{sobolevinterpolation} to write
	\begin{align*}
		\Bigl|\int_{\TT^d} Y_N\Ups_N\, \rd x\Bigr|^\gamma &\lesssim \norm{Y_N}{\cC^{\alpha-\frac d2-\eps}}^\gamma\norm{\Ups_N}{H^{2\eps}}^\gamma \\
		&\lesssim C(\delta)\norm{Y_N}{\cC^{\alpha-\frac d2-\eps}}^{\gamma c(\eps)}+\delta\norm{\Ups_N}{H^{2\eps}}^{\gamma+\eps} \\
		&\lesssim C(\delta)\norm{Y_N}{\cC^{\alpha-\frac d2-\eps}}^{\gamma c(\eps)}+\delta\norm{\Ups_N}{L^2}^{(\gamma+\eps)(1-\frac{2\eps}\alpha)}\norm{\Ups_N}{H^\alpha}^{(\gamma+\eps)\frac{2\eps}\alpha};
	\end{align*}
	using Young's inequality, we now obtain \eqref{tamingtermbdeasy}, and and thus complete the proof of Lemma~\ref{b-dcrossterms}.
\end{proof}
\begin{proof}[Proof of Lemma~\ref{phi33heart}.]
	On the event $\{\norm{\Ups_N}{L^2}^2>|\int_{\TT^d}Y_N\Ups_N\, \rd x|\}$ we have
	\begin{equation*}
		\frac12\norm{\Ups_N}{L^2}^2 \le \Big|\int_{\TT^d}(2Y_N\Ups_N+\Ups_N^2)\, \rd x\Big| \le \frac32\norm{\Ups_N}{L^2}^2,
	\end{equation*}
	and so we obtain the desired conclusion. Hereafter we work on the event $\{\norm{\Ups_N}{L^2}^2\le |\int_{\TT^d}Y_N\Ups_N\, \rd x|\}$. Define frequency projectors $\Pi_1=\1\{|\nabla|\le 2\}$ and $\Pi_j=\1\{2^{j-1}<|\nabla|\le 2^j\}$ for $j\ge 2$; set $\Pi_{\le j}=\sum_{k=1}^j \Pi_k$ and $\Pi_{>j}={\rm id}-\Pi_{\le j}$. We use $L^2$-projections of $\Ups_N$ onto $\Pi_jY_N$:
	\begin{equation*}
		\Ups_N=\sum_{j=1}^\infty (\lambda_j\Pi_jY_N+w_j),
	\end{equation*}
	where
	\begin{equation*}
		\lambda_j = 
		\begin{cases} 
			\frac{\innerprod{\Ups_N}{\Pi_jY_N}{L^2}}{\norm{\Pi_jY_N}{L^2}^2}, & \text{if}\ \norm{\Pi_jY_N}{L^2}\ne 0, \\ 
			0, & \text{otherwise}; 
		\end{cases} 
		\qquad w_j=\Pi_j\Ups_N-\lambda_j\Pi_jY_N.
	\end{equation*}
	Following from this orthogonal decomposition, we have
	\begin{align*}
		\norm{\Ups_N}{L^2}^2 &=\sum_{j=1}^\infty(\lambda_j^2\norm{\Pi_j\Ups_N}{L^2}^2+\norm{w_j}{L^2}^2), \\
		\int_{\TT^d}Y_N\Ups_N\, \rd x&=\sum_{j=1}^\infty \lambda_j\norm{\Pi_jY_N}{L^2}^2.
	\end{align*}
	As $\norm{w_j}{L^2}^2\ge 0$, we have
	\begin{equation*}
		\sum_{j=1}^\infty \lambda^2 \norm{\Pi_jY_N}{L^2}^2\le C\Big|\sum_{j=1}^\infty \lambda_j \norm{\Pi_jY_N}{L^2}^2\Big|.
	\end{equation*}
	We now work to bound the right-hand-side of the above, with the idea of decomposing the sum into high and low frequencies. Namely, fix $j_0$ (to be chosen later), noting that, since $|\lambda_j|\le \frac{\norm{\Pi_j\Ups_N}{L^2}}{\norm{\Pi_jY_N}{L^2}}$ by the Cauchy-Schwarz inequality, we have
	\begin{align*}
		\Big|\sum_{j>j_0}\lambda_j\norm{\Pi_jY_N}{L^2}^2\Big| &\le \Bigl(\sum_{j=1}^\infty 2^{2\alpha j}\lambda_j^2\norm{\Pi_jY_N}{L^2}^2\Bigr)^\frac12\Bigl(\sum_{j>j_0}2^{-2\alpha j}\norm{\Pi_jY_N}{L^2}^2\Bigr)^\frac12 \\
		&\le \Bigl(\sum_{j=1}^\infty 2^{2\alpha j}\norm{\Pi_j\Ups_N}{L^2}^2\Bigr)^\frac12\Bigl(\sum_{j>j_0}2^{-2\alpha j}\norm{\Pi_jY_N}{L^2}^2\Bigr)^\frac12 \\
		&\lesssim \norm{\Ups_N}{H^\alpha}\norm{\Pi_{>j_0}Y_N}{H^{-\alpha}},
	\end{align*}
	where we used the Littlewood-Paley characterisation of Sobolev norms for the last line. On the other hand,
	\begin{align*}
		\Big|\sum_{j\le j_0} \lambda_j\norm{\Pi_jY_N}{L^2}^2\Big| &\le \Bigl(\sum_{j=1}^\infty\lambda_j^2\norm{\Pi_jY_N}{L^2}^2\Bigr)^{\frac12}\Bigl(\sum_{j\le j_0}\norm{\Pi_jY_N}{L^2}^2\Bigr)^{\frac12} \\
		&\le C^{\frac12}\Big|\sum_{j=1}^\infty \lambda_j \norm{\Pi_jY_N}{L^2}^2\Big|^{\frac12}\Bigl(\sum_{j\le j_0}\norm{\Pi_jY_N}{L^2}^2\Bigr)^{\frac12} \\
		&\le \frac12\Big|\sum_{j=1}^\infty \lambda_j \norm{\Pi_jY_N}{L^2}^2\Big|+C'\norm{\Pi_{\le j_0}Y_N}{L^2}^2,
	\end{align*}
	using Young's inequality for the last line. It follows that
	\begin{equation*}
		\Big|\sum_{j=1}^\infty \lambda_j \norm{\Pi_jY_N}{L^2}^2\Big| \lesssim \norm{\Ups_N}{H^\alpha}\norm{\Pi_{>j_0}Y_N}{H^{-\alpha}}+\norm{\Pi_{\le j_0}Y_N}{L^2}^2.
	\end{equation*}
	We now work to bound the terms $\norm{\Pi_{>j_0}Y_N}{H^{-\alpha}}$ and $\norm{\Pi_{\le j_0}Y_N}{L^2}^2$. First observe that, as $\EE[(\ja{\nabla}^{-\alpha}\Pi_{>j_0}Y_N)(x)^2]$ is independent of $x\in\TT^d$, we have
	\begin{align*}
		\norm{\Pi_{>j_0}&Y_N}{H^{-\alpha}}^2 \\
		&= \int_{\TT^d} \wick{(\ja{\nabla}^{-\alpha}\Pi_{>j_0}Y_N)^2}\, \rd x+\EE[(\ja{\nabla}^{-\alpha}\Pi_{>j_0}Y_N(x_0))^2] \\
		&\le 2^{-aj_0}\Bigl(\sum_{j=1}^\infty 2^{2aj}\Bigl(\int_{\TT^d}\wick{(\ja{\nabla}^{-\alpha}\Pi_{>j}Y_N)^2}\, \rd x\Bigr)^2\Bigr)^{\frac12}+\EE[(\ja{\nabla}^{-\alpha}\Pi_{>j_0}Y_N(x_0))^2]
	\end{align*}
	for some $x_0\in\TT^d$. Let the right-hand-side of the display above be $2^{-aj_0}B_1+\tilde\sigma_{>j_0}$. Now, by first using Minkowski's integral inequality, followed by the hypercontractive estimate Lemma~\ref{hypercontractivity}, we have
	\begin{equation} \label{B1mombd} \begin{aligned}
		\EE B_1^p &\le \Bigl(\sum_{j=1}^\infty\bignorm{2^{aj}\int_{\TT^d}\wick{(\ja{\nabla}^{-\alpha}\Pi_{>j}Y_N)^2}\, \rd x}{L^p(\PP)}^2\Bigr)^{\frac{p}{2}} \\
		&\le \Bigl(\sum_{j=1}^\infty (p-1)^2\bignorm{2^{aj}\int_{\TT^d}\wick{(\ja{\nabla}^{-\alpha}\Pi_{>j}Y_N)^2}\, \rd x}{L^2(\PP)}^2\Bigr)^{\frac{p}{2}} 
	\end{aligned} \end{equation} 
	for any finite $p\ge 2$ (and hence $p\ge 1$). Next, using Hermite orhogonality (Lemma~\ref{hermiteorthogonality}), we have
	\begin{align*}
		\EE\Bigl[\Bigl(\int_{\TT^d}\wick{&\,\,(\ja{\nabla}^{-\alpha}\Pi_{>j}Y_N)^2}\, \rd x\Bigr)^2\Big]  \\ &= \int_{\TT^d\times\TT^d}\EE[H_2(\ja{\nabla}^{-\alpha}\Pi_{>j}Y_N(x);\tilde\sigma_{>j})H_2(\ja{\nabla}^{-\alpha}\Pi_{>j}Y_N(y);\tilde\sigma_{>j})]\, \rd x\, \rd y \\
		&=\int_{\TT^d\times\TT^d} 2(\EE[(\ja{\nabla}^{-\alpha}\Pi_{>j}Y_N(x)(\ja{\nabla}^{-\alpha}\Pi_{>j}Y_N(y))])^2\, \rd x\, \rd y \\
		&= 2\sum_{2^j<|n|\le N}\frac{1}{\ja{n}^{8\alpha}} \\
		&\lesssim 2^{-(8\alpha-d)j}.
	\end{align*}
	In particular, using the results of the above display with \eqref{B1mombd}, we have
	\begin{equation*}
		\EE B_1^p\lesssim p^p\Bigl(\sum_{j=1}^\infty 2^{-(8\alpha-d-2a)j}\Bigr)^{\frac{p}{2}},
	\end{equation*}
	and this is essentially bounded by $p^p$. Here we must assume $8\alpha>d+2a$. Moreover $\tilde\sigma_{>j_0}\sim2^{-(4\alpha-d)j_0}$. Analogously to the previous computation we have
	\begin{align*}
		\norm{\Pi_{\le j_0}Y_N}{L^2}^2 &=\int_{\TT^d}\wick{(\Pi_{\le j_0}Y_N)^2}\, \rd x+\EE[(\Pi_{\le j_0}Y_N(x_0))^2] \\ 
		&\lesssim \sum_{j=1}^\infty \Big|\int_{\TT^d}\wick{(\Pi_jY_N)^2}\, \rd x\Big|+\EE[(\Pi_{\le j_0}Y_N(x_0))^2],
	\end{align*}
	where, to obtain the last estimate we use the fact that $\Pi_{\le j_0}=\Pi_1+\cdots+\Pi_{j_0}$ and the multinomial expansion for the Wick power. Label the right-hand-side above by $B_2+\tilde\sigma_{\le j_0}$. Working analogously to the computations done above, we have
	\begin{align*}
		\EE B_2^p &\le \Bigl(\sum_{j=1}^\infty \bignorm{\int_{\TT^d}\wick{(\Pi_jY_N)^2}\, \rd x}{L^p(\PP)}\Bigr)^p \\
		&\le \Bigl(\sum_{j=1}^\infty (p-1)\bignorm{\int_{\TT^d}\wick{(\Pi_jY_N)^2}\, \rd x}{L^2(\PP)}\Bigr)^p \\
		&\lesssim p^p\Bigl(\sum_{j=1}^\infty\sum_{2^{j-1}<|n|\le 2^j}\frac{1}{\ja{n}^{4\alpha}}\Bigr)^p,
	\end{align*}
	and, as before, this is essentially bounded by $p^p$. Also $\tilde\sigma_{\le j_0}\lesssim2^{(d-2\alpha)j_0}$. Altogether now, after several applications of Young's inequality, we have
	\begin{align*}
		\norm{\Ups_N}{L^2}^2 &\lesssim \Big|\int_{\TT^d} Y_N\Ups_N\, \rd x\Big| \\
		&= \Big|\sum_{j=1}^\infty \lambda_j\norm{\Pi_jY_N}{L^2}^2\Big| \\
		&\lesssim \norm{\Ups_N}{H^\alpha}(\norm{\Pi_{>j_0}Y_N}{H^{-\alpha}}^2)^{\frac12}+\norm{\Pi_{\le j_0}Y_N}{L^2}^2 \\
		&\lesssim \norm{\Ups_N}{H^\alpha}(2^{-\frac{aj_0}{2}}B_1^{\frac12}+2^{-\frac{(4\alpha-d)j_0}{2}})+B_2+2^{(d-2\alpha)j_0}.
	\end{align*}
	Now, taking $j_0$ so that $2^{j_0}\sim1+\norm{\Ups_N}{H^\alpha}^{\frac2d}$ (chosen so that $\norm{\Ups_N}{H^\alpha}2^{-\frac{(4\alpha-d)j_0}{2}}\sim 2^{(d-2\alpha)j_0}$) and applying Young's inequality, the final quantity above can be bounded up to a constant by 
	\begin{equation*}
		\norm{\Ups_N}{H^\alpha}^{(1-\frac ad)+\eps}+\norm{\Ups_N}{H^\alpha}^{\frac{2d-4\alpha}{d}}+B_1^{c(\eps)}+B_2
	\end{equation*}
	for small $\eps$ and large $c(\eps)$. We can choose $a,\eps$, so that the above is bounded by the right-hand-side of \eqref{phi33hearteq}. This completes the proof.
\end{proof}
\subsection{The regular and weakly nonlinear regimes.}
Next, we move to prove the properties of the $\Phi^3_d$ measures $\varrho$ and $\varrho^\delta$, namely, the continuity $\varrho\ll\mu$ in the regime $d<3\alpha$ and the singularity $\varrho\perp\mu$ when $d=3\alpha$ (conditional on $|\sigma|$ small and the further renormalisation via the $\beta_N$). 

\begin{prop}
	Let $2\alpha<d<3\alpha$. Then, as measures on $\cC^{\alpha-\frac d2-\eps}$, we have $\varrho=\varrho^\delta$ and both are absolutely continuous with respect to $\mu$.
\end{prop}
\begin{proof}
	By the uniform exponential integrability Proposition~\ref{uniexpint} and dominated convergence applied to $(V_N)$, we have 
	\begin{equation}
		\varrho(\rd u) = \cZ^{-1}\exp\Bigl(\frac\sigma3\int_{\TT^d}\wick{u^3}\, \rd x-A\Big|\int_{\TT^d}\wick{u^2}\, \rd x\Big|^\gamma\Bigr)\, \mu(\rd u).
	\end{equation}
	Next, we show that $\varrho^\delta$ is a probability. Let $(\varphi_\eps)$ be as in Proposition~\ref{asfiniteness}, so that
	\begin{align*}
		\varrho^\delta(\cD') &\le \lim_{L\to\infty}\liminf_{\eps\downarrow0}\lim_{N\to\infty} \int_{\cD'}\exp(\delta\min\{\norm{\varphi_\eps\ast u_N}{\cA}^q,L\})\, \vartheta_{N,\delta}(\rd u) \\
		&\le \liminf_{L\to\infty}\limsup_{N\to\infty}\int_{\cD'}\exp(\delta\min\{\norm{u_N}{\cA}^q,L\}-\delta\norm{u_N}{\cA}^q-V_N(u_N))\, \mu(\rd u) \\
		&\le\limsup_{N\to\infty} \varrho_N(\cD'),
	\end{align*}
	which is finite by Proposition~\ref{uniexpint}. In particular we can normalise $\varrho^\delta$. Moreover as above
	\begin{equation}
		\vartheta_{N,\delta}(\rd u)\rightharpoonup\exp\Bigl(-\delta\norm{u}{\cA}^q+\frac\sigma3\int_{\TT^d}\wick{u^3}\, \rd u-A\Big|\int_{\TT^d}\wick{u^2}\, \rd x\Big|^\gamma\Bigr)\, \mu(\rd u),
	\end{equation}
	from which we may conclude $\varrho=\varrho^\delta$.
\end{proof}
\begin{prop}
	Let $d=3\alpha$ and let $\sigma$ be sufficiently small. Then, as measures on $\cC^{\alpha-\frac d2-\eps}$, we have $\varrho\perp\mu$.
\end{prop}
\begin{proof}
	We will prove that there exists an increasing sequence $(N_k)$ of positive integers such that the set
	\begin{equation*}
		S=\{u\in \cD':(\log N_k)^{-\frac34}(V_{N_k}(u)-\beta_{N_k})=0\}
	\end{equation*}
	has $\mu(S)=1$ but $\varrho(S)=0$, from which the proposition follows. To this end, by \eqref{masstamingpotential} and \eqref{hypercontractivity}, we have
	\begin{align*}
		\norm{V_N-\beta_N}{L^2(\mu)}^2 &\lesssim_{\sigma,A} \bignorm{\int_{\TT^d}\wick{u_N^3}\, \rd x}{L^2(\mu)}^2+\bignorm{\int_{\TT^d}\wick{u_N^2}\, \rd x}{L^6(\mu)}^6 \\
		&\lesssim \bignorm{\int_{\TT^d}\wick{u_N^3}\, \rd x}{L^2(\mu)}^2+\bignorm{\int_{\TT^d}\wick{u_N^2}\, \rd x}{L^2(\mu)}^6.
	\end{align*}
	Use \eqref{hermiteorthogonalityeq} in Lemma~\ref{hermiteorthogonality} to compute
	\begin{align*}
		\bignorm{\int_{\TT^d}\wick{u_N^k}\, \rd x}{L^2(\mu)}^2 &= \EE\Bigl[\int_{\TT^d\times\TT^d}H_k(Y_N(x);\sigma_N)H_k(Y_N(y);\sigma_N)\, \rd x\, \rd y\Big] \\
		&= \int_{\TT^d\times\TT^d} (\EE[Y_N(x)Y_N(y)])^k\, \rd x\, \rd y \\
		&= \int_{\TT^d\times\TT^d} \Bigl(\sum_{|n|,|m|\le N}\frac{\EE[B_n(1)B_m(1)]}{\ja{n}^\alpha\ja{m}^\alpha}\re^{2\pi\ri(n\cdot x+m\cdot y)}\Bigr)^k\, \rd x\, \rd y \\
		&=\int_{\TT^d\times\TT^d} \sum_{|n_j|\le N\,j=1,\ldots, k}\frac{1}{\ja{n_1}^{2\alpha}\cdots\ja{n_k}^{2\alpha}}\re^{2\pi\ri(n_1+\cdots+n_k)\cdot(x-y)}\, \rd x\, \rd y \\
		&= \sum_{\substack{n_1+\cdots+n_k=0\\ |n_j|\le N,\, j=1,\ldots, k}}\frac{1}{\ja{n_1}^{2\alpha}\cdots\ja{n_k}^{2\alpha}}.
	\end{align*}
	Hence it follows by \eqref{discrconveq} in Lemma~\ref{discrconv} that
	\begin{align*}
		\bignorm{\int_{\TT^d}\wick{u_N^2}\, \rd x}{L^2(\mu)}^2&\lesssim 1 \\
		\bignorm{\int_{\TT^d}\wick{u_N^3}\, \rd x}{L^2(\mu)}^2&\lesssim \sum_{|n|\le N}\frac1{\ja{n}^{2\alpha}}\sum_{\substack{n_1+n_2=n}}\frac{1}{\ja{n_1}^{2\alpha}\ja{n_2}^{2\alpha}} \\
		&\lesssim \sum_{|n|\lesssim N}\frac{1}{\ja{n}^{3\alpha}},
	\end{align*}
	which is comparable to $\log N$ when $d=3\alpha$. It follows that
	\begin{equation*}
		\lim_{N\to\infty}\norm{(\log N)^{-\frac34}(V_N-\beta_N)}{L^2(\mu)}=0.
	\end{equation*}
	Next, we will prove that
	\begin{equation*}
		\lim_{N\to\infty}\norm{\exp((\log N)^{-\frac34}(V_N-\beta_N))}{L^1(\varrho)}=0.
	\end{equation*}
	Arguing along subsequences, this furnishes a subsequence $(N_k)$ to be used in the definition of $S$. Write $\tilde V_N=(\log N)^{-\frac34}(V_N-\beta_N)$. Now, letting $\chi$ be as in Proposition~\ref{asfiniteness}, and using the weak convergence $\varrho_N\rightharpoonup\varrho$, we have
	\begin{align*}
		\int_{\cD'} &\exp(\tilde V_N(u))\, \varrho(\rd u)  \\
	&\le \liminf_{M\to\infty}\int_{\cD'} \exp(\tilde V_N(u))\chi\Bigl(\frac{\tilde V_N(u)}{M}\Bigr)\, \varrho(\rd u) \\
		&\le \liminf_{M\to\infty}\lim_{K\to\infty} \cZ^{-1}\int_{\cD'} \exp(\tilde V_N(u))\exp(-V_K(u))\chi\Bigl(\frac{\tilde V_N(u)}{M}\Bigr)\, \mu(\rd u) \\
		&\le\limsup_{K\to\infty} \cZ^{-1}\int_{\cD}\exp(\tilde V_N(u)-V_K(u))\, \mu(\rd u),
	\end{align*}
	where $\cZ=\lim\,\cZ_N$. Applying the Bou\'e-Dupuis formula with our change-of-variable, we will be interested in the limit as $N\to\infty$ of the quantity below, where $K\ge N$:
	\begin{align*}
		\inf_{\dot\Ups^K\in\HH_{\rm a}^\alpha}&\EE\Bigl[-(\log N)^{-\frac34}(V_N(Y+\Ups_K+\sigma\fZ_K)-\beta_N)+V_K(Y+\Ups_K+\sigma\fZ_K) \\ 
		&\hskip200pt+\frac12\int_0^1\norm{\dot\Ups^K(t)}{H^\alpha}^2\, \rd t\Big];
	\end{align*}
	in particular we will show that the above tends to infinity as $N\to\infty$. Let $\cE$ be as in Proposition~\ref{uniqueness}; picking appropriate constants in \eqref{ZNb-dbd2}, we can show that
	\begin{equation*}
		\EE\Bigl[V_K(Y+\Ups_K+\sigma\fZ_K)+\frac12\int_0^1\norm{\dot\Ups^K(t)}{H^\alpha}^2\, \rd t\Big] \ge \frac1{10}\cE(\dot\Ups^K)-C
	\end{equation*}
	for some large $C$. Expanding for $K\ge N$, we have
	\begin{align*}
		V_N(Y+\Ups_K+\sigma\fZ_K)-\beta_N=&-\frac\sigma3\int_{\TT^d}\wick{Y_N^3}\, \rd x-\sigma\int_{\TT^d}\wick{Y_N^2}\Theta_N\, \rd x \\
		&-\sigma\int_{\TT^d}Y_N\Theta_N^2\, \rd x-\frac\sigma3\int_{\TT^d}\Theta_N^3\, \rd x \\
		&+A\Big|\int_{\TT^d}\wick{(Y_N+\Theta_N)^2}\, \rd x\Big|^3.
	\end{align*}
	The first term vanishes under expectation and, again by choosing appropriate constants in \eqref{ZNb-dbd2}, the final three terms are bounded by $1+\cE(\dot\Ups^K_N)$. Aiming to relate this quantity to $\cE(\dot\Ups^K)$, write
	\begin{align*}
		\EE\Bigl[\Big|\int_{\TT^d}(2Y_N\Ups^K_N+(\Ups^K_N)^2)\, \rd x\Big|^3\Big]&\lesssim \bignorm{\int_{\TT^d} Y_N\Ups^K_N\, \rd x}{L^3(\PP)}^3+\norm{\Ups^K_N}{L^6(\PP;L^2)}^6 \\
		&\lesssim 1+\norm{\Ups_K}{L^6(\PP;L^2)}^6+\norm{\Ups^K}{L^2(\PP;H^\alpha)}^2 \\
		&\lesssim 1+\cE(\dot\Ups^K).
	\end{align*}
	Hence for $K\ge N\gg1$ we have
	\begin{align*}
		\inf_{\dot\Ups^K\in\HH_{\rm a}^\alpha}&\EE\Bigl[-(\log N)^{-\frac34}(V_N(Y+\Ups_K+\sigma\fZ_K)-\beta_N)+V_K(Y+\Ups_K+\sigma\fZ_K) \\ 
		&\hskip200pt+\frac12\int_0^1\norm{\dot\Ups^K(t)}{H^\alpha}^2\, \rd t\Big] \\
		&\ge \inf_{\dot\Ups^K\in\HH_{\rm a}^\alpha}\Bigl[\EE\Bigl[\sigma(\log N)^{-\frac34}\int_{\TT^d}\wick{Y_N^2}\Theta_N\, \rd x\Big]+\frac1{20}\cE(\dot\Ups^K)\Big]-C
	\end{align*}
	By Lemma~\ref{pathwreg}, one has $\innerprod{\dot\fZ^N(t)}{\dot\fZ_N(t)}{H^\alpha}\sim t^2\log N$, and so we compute
	\begin{align*}
		\sigma\EE&\Bigl[\int_{\TT^d}\wick{Y_N^2}\Theta_N\, \rd x\Big]  \\
		&= \sigma\EE\Bigl[\int_0^1\int_{\TT^d}\wick{Y_N^2(t)}\dot\Theta_N(t)\, \rd x\, \rd t\Big] \\
		&= \sigma\EE\Bigl[\int_0^1\innerprod{\dot\fZ^N(t)}{\dot\Ups^K_N(t)}{H^\alpha}\, \rd t\Big]+\sigma^2\EE\Bigl[\int_0^1\innerprod{\dot\fZ^N(t)}{\dot\fZ_N(t)}{H^\alpha}\, \rd t\Big] \\
		& \ge -\eps\EE\Bigl[\int_0^1\norm{\,\wick{Y_N^2(t)}\,}{H^{-\alpha}}\, \rd t\Big]-C_\eps\EE\Bigl[\int_0^1\norm{\dot\Ups^K_N(t)}{H^\alpha}\rd t\Big]+C\log N
	\end{align*}
	In particular with $K\ge N\gg1$, we have
	\begin{align*}
		\inf_{\dot\Ups^K\in\HH_{\rm a}^\alpha}&\EE\Bigl[-(\log N)^{-\frac34}(V_N(Y+\Ups_K+\sigma\fZ_K)-\beta_N)+V_K(Y+\Ups_K+\sigma\fZ_K) \\ 
		&\hskip200pt+\frac12\int_0^1\norm{\dot\Ups^K(t)}{H^\alpha}^2\, \rd t\Big] \\
		&\ge \inf_{\dot\Ups^K\in\HH_{\rm a}^\alpha}\Bigl[C(\log N)^{\frac14}+\frac1{40}\cE_K(\dot\Ups^K)\Big]-C
	\end{align*}
	As $\cE_K$ is nonnegative, taking a limit in $N\to\infty$ above allows us to conclude the proof.
\end{proof}
\subsection{The critical and strongly nonlinear regime.}
In this section we prove the non-normalisability of $\varrho^\delta$, and the non-convergence of the $\varrho_N$ in the critical and strongly nonlinear regime.

\begin{prop} \label{nonnormalisability}
	Let $d=3\alpha$. There exists $\sigma_1\gg1$ such that, when $|\sigma|\ge \sigma_1$, we have
	\begin{equation}
		\varrho^\delta(\cD')=\infty.
	\end{equation}
\end{prop}
\begin{proof}
	Let $(\varphi_\eps)$ be as in Proposition~\ref{asfiniteness}, and compute, using the weak convergence of $(\vartheta_{N,\delta})$, that 
	\begin{align*}
		\int_{\cD'} \exp(\delta\norm{u}{\cA}^q)\, \vartheta_\delta(\rd u) &\ge \int_{\cD'} \exp(\delta\norm{\varphi_\eps\ast u}{\cA}^q)\, \vartheta_\delta(\rd u) \\
		&\ge \lim_{L\to\infty}\lim_{N\to\infty} \int_{\cD'} \exp(\delta\min\{\norm{\varphi_\eps\ast u_N}{\cA}^q,L\})\, \vartheta_{N,\delta}(\rd u).
	\end{align*}
	In particular, it suffices to prove that
	\begin{equation*}
		\lim_{L\to\infty}\lim_{N\to\infty}\EE[\exp(\delta\min\{\norm{\varphi_\eps\ast Y_N}{\cA}^q\}-\delta\norm{Y_N}{\cA}^q-V_N(Y_N))] = \infty.
	\end{equation*}
	Using the Bou\'e-Dupuis formula, the expectation above is equal to
	\begin{equation} \label{b-dnn} \begin{aligned}
	\inf_{\dot\Ups^N\in\HH_{\rm a}^\alpha}\EE\Bigl[&-\delta\min\{\norm{\varphi_\eps\ast (Y_N+\Theta_N)}{\cA}^q,L\}+\delta\norm{Y_N+\Theta_N}{\cA}^q \\
		&-\sigma\int_{\TT^d}Y_N\Theta_N^2\, \rd x-\frac\sigma3\int_{\TT^d}\Theta_N^3\, \rd x+A\Big|\int_{\TT^d}\wick{(Y_N+\Theta_N)^2}\, \rd x\Big|^\gamma \\ 
		&\hskip200pt+\frac12\int_0^1\norm{\dot\Ups^N}{H^\alpha}^2\, \rd t\Big].
	\end{aligned} \end{equation}
	In what follows, we approach as in \cite{OOT25,OOT24,OST24}, and aim to choose a drift term $\dot{\und\Ups}^N$ for which $\und\Ups_N$ resembles ``$-Y(1)$ plus a perturbation'', where the perturbation is bounded in $L^2$ but has large $L^3$ norm. \vspace{\baselineskip}

	We first construct our perturbation term. Fix $M\gg 1$. Let $f$ be a real-valued Schwartz function on $\RR^d$ such that its Fourier transform $\hat f$ is smooth, even, and non-negative, supported on $\{\frac12<|\xi|\le 1\}$, and with $\norm{f}{L^2(\RR^d)}=1$. Define $f_M$ on $\TT^d$ by
	\begin{equation}
		f_M(x) = M^{-\frac d2}\sum_{n\in\ZZ^d} \hat f\Bigl(\frac nM\Bigr)\re^{2\pi\ri n\cdot x}.
	\end{equation}
	Note that, by the Poisson summation formula and properties of the Fourier transform under dilation, we have
	\begin{equation}
		f_M(x) = \sum_{m\in\ZZ^d}M^{\frac d2} f(Mx+Mm).
	\end{equation}
	Moreover, we have the following estimates.
	\begin{lemma} \label{fMest}
		Let $c>0$ be any positive number.
		\begin{align}
			& \int_{\TT^d} f_M^2\, \rd x=1+O(M^{-c}), \label{fML2} \\
			& \int_{\TT^d} (\ja{\nabla}^{-c}f_M)^2\, \rd x\lesssim M^{-2c}, \label{fMH-c} \\
			& \int_{\TT^d} |f_M|^3\, \rd x\sim \int_{\TT^d} f_M^3\, \rd x\sim M^{\frac d2}. \label{fML3}
		\end{align}
	\end{lemma}
	\noindent We delay the proof of Lemma~\ref{fMest} until later. Next, we construct an approximation to $-Y(1)$. To this end, let
	\begin{equation}
		Z_M(x)=\sum_{n\in\ZZ^d}\hat{Y(\tfrac12)}\re^{2\pi\ri n\cdot x}=\sum_{n\in\ZZ^d}\frac{B_n(\frac12)}{\ja{n}^\alpha}\re^{2\pi \ri n\cdot x},
	\end{equation}
	noting that $Z_M$ is measurable in the natural filtration for the Brownian motions past time $t=\frac12$. Let $\kappa_M=\EE[Z_M(x)^2]$, noting that $\kappa_M$ is independent of $x\in\TT^d$. We have the following estimates for $Z_M$.
	\begin{lemma} \label{ZMest}
		Let $1\le p<\infty$ and $N\ge M$.
		\begin{align}
			& \kappa_M \sim M^{d-2\alpha}, \label{kappaM} \\
			& \EE\Bigl[\int_{\TT^d} |Z_M|^p\, \rd x\Big] \lesssim_p M^{\frac p2(d-2\alpha)}, \label{ZMLp} \\
			& \EE\Bigl[\Bigl(\int_{\TT^d} Z_M^2\, \rd x-\kappa_M\Bigr)^2\Big]+\EE\Bigl[\Bigl(\int_{\TT^d} Y_NZ_M\, \rd x-\int_{\TT^d} Z_M^2\, \rd x\Bigr)^2\Big]\lesssim 1, \label{ZMconst} \\
			& \EE\Bigl[\Bigl(\int_{\TT^d} Y_Nf_M\, \rd x\Bigr)^2\Big]+\EE\Bigl[\Bigl(\int_{\TT^d} Z_Mf_M\, \rd x\Bigr)^2\Big]\lesssim M^{-2\alpha}. \label{ZMsmall}
		\end{align}
	\end{lemma}
	\noindent We again delay the proof of Lemma~\ref{ZMest} until later. Now ready to define our drift, we set
	\begin{equation}
		\dot{\und\Ups}^N(t)=2\cdot \1\{t>\tfrac12\}(-Z_M+\sgn\,\sigma\sqrt{\kappa_M}f_M),
	\end{equation}
	so that
	\begin{equation} \label{undUpsN}
		\und\Ups_N=-Z_M+\sgn\,\sigma\sqrt{\kappa_M}f_M.
	\end{equation}
	We now approach \eqref{b-dnn} term-by-term. First observe that
	\begin{align*}
		-\delta\min\{\norm{\varphi_\eps\ast(Y_N&+\und\Theta_N)}{\cA}^q,L\}+\delta\norm{Y_N+\und\Theta_N}{\cA}^q \\
		&=-\delta\min\{\norm{\varphi_\eps\ast(Y_N+\und\Theta_N)}{\cA}^q-\norm{Y_N+\und\Theta_N}{\cA}^q,L-\norm{Y_N+\und\Theta_N}{\cA}^q\};
	\end{align*}
	we will bound each term in the minimum above separately. For the first, we make some preliminary observations. Note that the constraints on the definition of $\cA=B^{-2s}_{3,\infty}$ which arise in the proof of Proposition~\ref{uniexpint} provide $s>\frac\alpha2$ and so we can afford a Schauder estimate of the form 
	\begin{equation*}
		\norm{f_M}{\cA}\lesssim\norm{f_M}{H^{-\frac\alpha2}},
	\end{equation*}
	from which
	\begin{align*}
		\norm{\varphi_\eps\ast(Y_N+\und\Theta_N)}{\cA}^q-&\norm{Y_N+\und\Theta_N}{\cA}^q \\ 
		&\gtrsim -|\norm{(\varphi_\eps-\delta_0)\ast(Y_N+\und\Theta_N)}{\cA}\norm{Y_N+\und\Theta_N}{\cA}^q| \\
		&\gtrsim -\kappa_M^{\frac q2}\norm{(\varphi_\eps-\delta_0)\ast f_M}{H^{-\frac\alpha2}}^q-(\norm{Y_N}{\cA}^q+\norm{Z_M}{\cA}^q+|\sigma|\norm{\fZ_N}{\cA}^q) \\
		&\gtrsim-M^{\frac{q\alpha}{2}}\eps^{c_{\alpha,q}}-(\norm{Y_N}{\cA}^q+\norm{Z_M}{\cA}^q+|\sigma|\norm{\fZ_N}{\cA}^q) \\
		&\gtrsim -1-(\norm{Y_N}{\cA}^q+\norm{Z_M}{\cA}^q+|\sigma|\norm{\fZ_N}{\cA}^q),
	\end{align*}
	after choosing $\eps$ sufficiently small depending on $M$. The terms under parentheses are bounded under expectation. For the second term in the minimum we use the Schauder estimate above and Lemma~\ref{fMest}
	\begin{align*}
		-\norm{Y_N+\und\Theta_N}{\cA}^q &\gtrsim -\kappa_M^{\frac q2}\norm{f_M}{\cA}^q-(\norm{Y_N}{\cA}^q+\norm{Z_M}{\cA}^q+|\sigma|\norm{\fZ_N}{\cA}^q) \\
		&\gtrsim -M^{-\frac{q\alpha}2}-(\norm{Y_N}{\cA}^q+\norm{Z_M}{\cA}^q+|\sigma|\norm{\fZ_N}{\cA}^q) \\
		&\gtrsim -1-(\norm{Y_N}{\cA}^q+\norm{Z_M}{\cA}^q+|\sigma|\norm{\fZ_N}{\cA}^q).
	\end{align*}
	Next, we have, by embeddings and Young's inequality, the bound
	\begin{align*}
		-\sigma\int_{\TT^d} Y_N\und\Theta_N^2\, \rd x &\lesssim |\sigma|\norm{Y_N}{\cC^{-\frac\alpha2-\eps}}\norm{\und\Theta_N^2}{B_{1,1}^{\frac\alpha2+\eps}} \\
		&\lesssim |\sigma|\norm{Y_N}{\cC^{-\frac\alpha2-\eps}}\norm{\und\Theta_N}{B_{2,1}^{\frac\alpha2+\eps}}\norm{\und\Theta_N}{L^2} \\
		&\lesssim |\sigma|^6\norm{Y_N}{\cC^{-\frac\alpha2-\eps}}^6+\norm{\und\Theta_N}{H^{\frac\alpha2+2\eps}}^2+\norm{\und\Theta_N}{L^2}^3 \\
		&\lesssim \norm{\und\Ups_N}{H^\alpha}^2+\norm{\und\Ups_N}{L^2}^3+|\sigma|^6(\norm{Y_N}{\cC^{-\frac\alpha2-\eps}}^6+\norm{\fZ_N}{\cC^{\alpha-\eps}}^2);
	\end{align*}
	the rightmost terms are bounded under expectation, whereas
	\begin{equation*}
		\EE\norm{\und\Ups_N}{L^2}^2 \lesssim \EE\norm{Z_M}{L^2}^2+\kappa_M\norm{f_M}{L^2}^2\lesssim M^\alpha
	\end{equation*}
	using Lemmas \ref{fMest} and \ref{ZMest}; since $f_M$ and $Z_M$ have frequency support in $\{|n|\le M\}$, we have $\norm{\und\Ups_N}{H^\alpha}^2\lesssim M^{2\alpha}\norm{\und\Ups_N}{L^2}^2$, from which, using also a Wiener chaos estimate,
	\begin{equation*}
		\EE\Bigl[-\sigma\int_{\TT^d}Y_N\und\Theta_N^2\, \rd x\Big] \lesssim M^{2\alpha}M^\alpha+(M^\alpha)^{\frac32}+|\sigma| \lesssim M^{3\alpha}.
	\end{equation*}
	Moving on, using \eqref{undUpsN} and Young's inequality, we have
	\begin{align*}
		-\frac\sigma3\int_{\TT^d}\und\Theta_N^3\, \rd x &= \frac\sigma3\int_{\TT^d}Z_M^3\, \rd x-\frac{|\sigma|}3\kappa_M^{\frac32}\int_{\TT^d}f_M^3\, \rd x-\frac{\sigma^4}3\int_{\TT^d}\fZ_N^3\, \rd x \\
		&\hskip10pt-|\sigma|\sqrt{\kappa_M}\int_{\TT^d}Z_M^2f_M\, \rd x-\sigma^2\int_{\TT^d}Z_M\fZ_N\, \rd x \\
		&\hskip10pt+\sigma\kappa_M\int_{\TT^d}Z_Mf_M^2\, \rd x-\sigma^2\kappa_M\int_{\TT^d}f_M^2\fZ_N\, \rd x \\
		&\hskip10pt+\sigma^3\int_{\TT^d}Z_M\fZ_N^2\, \rd x-\sigma^3\sqrt{\kappa_M}\int_{\TT^d}f_M\fZ_N^2\, \rd x \\
		&\hskip10pt+2\sigma^2\sgn\,\sigma\sqrt{\kappa_M}\int_{\TT^d}Z_Mf_M\fZ_N\, \rd x \\
		&\le -\frac{|\sigma|}3\kappa_M^{\frac32}\int_{\TT^d}f_M^3\, \rd x+\eta|\sigma|\kappa_M^{\frac32}\int_{\TT^d}f_M^3\, \rd x \\
		&\hskip10pt+C_\eta\Bigl(|\sigma|\int_{\TT^d}|Z_M|^3\, \rd x+|\sigma|^4\int_{\TT^d}|\fZ_N|^3\, \rd x\Bigr)
	\end{align*}
	for any $0<\eta<1$; in particular, picking e.g. $\eta=\frac12$ and using Lemmas \ref{fMest} and \ref{ZMest}, we have
	\begin{equation*}
		\EE\Bigl[-\frac\sigma3\int_{\TT^d}\und\Theta_N^3\, \rd x\Big] \lesssim -|\sigma|M^{\frac{3\alpha}2}M^{\frac{3\alpha}2}+|\sigma|M^{\frac{3\alpha}2}+|\sigma|^4\lesssim -|\sigma|M^{3\alpha}.
	\end{equation*}
	Moving on, using a Wiener chaos estimate and expanding Wick powers, we have
	\begin{equation} \label{AYN+undThetaNexp} \begin{aligned}
		\EE\Bigl[A\Big|\int_{\TT^d}\wick{(Y_N+\und\Theta_N)^2}\, \rd x\Big|^\gamma\Big] &\lesssim_{A,\gamma} \Bigl(\EE\Bigl[\Big|\int_{\TT^d} (\wick{Y_N^2}+2Y_N\und\Theta_N+\und\Theta_N^2)\, \rd  x\Big|^2\Big]\Bigr)^{\frac\gamma2} \\
		&\lesssim \Bigl(\EE\Bigl[\Big|\int_{\TT^d} \wick{Y_N^2}\, \rd x+2\sigma\int_{\TT^d}Y_N\fZ_N\, \rd x+\sigma^2\int_{\TT^d}\fZ_N^2\, \rd x\Big|^2\Big] \\
		&\hskip15pt +\EE\Bigl[\Big|2\sigma\int_{\TT^d}\und\Ups_N\fZ_N\, \rd  x\Big|^2\Big] \\
		&\hskip15pt +\EE\Bigl[\Big|\int_{\TT^d}(2Y_N\und\Ups_N+\und\Ups_N^2)\, \rd x\Big|^2\Big]\Bigr)^{\frac\gamma2}.
	\end{aligned} \end{equation}
	The first expectation on the right-hand side above is bounded uniformly in $N\ge M\gg 1$ using arguments analogous to those which have appeared before. For the second expectation in \eqref{AYN+undThetaNexp}, we observe that
	\begin{align*}
		\Big|\int_{\TT^d} \und\Ups_N\fZ_N\, \rd x\Big|^2 &= \Bigl(\int_0^1 \Big|\int_{\TT^d} \ja{\nabla}^{-\alpha+\eps}\und\Ups_N\ja{\nabla}^{-\alpha-\eps}\pi_N\wick{Y_N^2(t)}\, \rd x \Big|\, \rd t\Bigr)^2 \\
		&\le \norm{\und\Ups_N}{H^{-\alpha+\eps}}^2\int_0^1 \norm{\pi_N\wick{Y_N^2(t)}}{H^{-\alpha-\eps}}^2\, \rd t \\
		&\lesssim \norm{\und\Ups_N}{H^{-\alpha+\eps}}^4+\int_0^1\norm{\pi_N\wick{Y_N^2(t)}}{H^{-\alpha-\eps}}^4\, \rd t
	\end{align*}
	using Jensen's and Young's inequalities. The second term above is bounded uniformly in $N$ and $t$ under expectation while, for the first, we use the Wiener chaos estimate and Lemmas \ref{fMest} and \ref{ZMest} to obtain
	\begin{align*}
		\EE\norm{\und\Ups_N}{H^{-\alpha+\eps}}^4 &\lesssim (\EE\norm{Z_M}{H^{-\alpha+\eps}}^2)^2+\kappa_M^2\norm{f_M}{H^{-\alpha+\eps}}^4 \\
		&\lesssim \Bigl(\sum_{|n|\le M}\ja{n}^{-4\alpha+2\eps}\Bigr)^2+M^{2\alpha}(M^{-2\alpha+2\eps})^2 \\
		&\lesssim 1.
	\end{align*}
	We now bound the third expectation in \eqref{AYN+undThetaNexp}. By expanding and grouping terms, we have
	\begin{align*}
		\EE\Bigl[&\Big|\int_{\TT^d}(2Y_N\und\Ups_N+\und\Ups_N^2)\, \rd x\Big|^2\Big] \\
		&= \EE\Bigl[\Big|-2\int_{\TT^d} Y_NZ_M\, \rd x+2\sqrt{\kappa_M}\int_{\TT^d} Y_Nf_M\, \rd x+\int_{\TT^d}Z_M^2\, \rd x \\
		&\hskip25pt-2\sqrt{\kappa_M}\int_{\TT^d} Z_Mf_M\, \rd x+\kappa_M\int_{\TT^d}f_M^2\, \rd x\Big|^2\Big] \\
		&\lesssim \EE\Bigl[\Bigl(\int_{\TT^d} Z_M^2\, \rd x-\kappa_M\Bigr)^2\Big]+\EE\Bigl[\Bigl(\int_{\TT^d} Y_NZ_M\, \rd x-\int_{\TT^d} Z_M^2\, \rd x\Bigr)^2\Big] \\
		&\hskip15pt+\kappa_M^2\Bigl(\int_{\TT^d} f_M^2\, \rd x-1\Bigr)^2+\kappa_M\EE\Bigl[\Bigl(\int_{\TT^d}Y_Nf_M\, \rd x\Bigr)^2\Big]+\kappa_M\EE\Bigl[\Bigl(\int_{\TT^d} Z_Mf_M\, \rd x\Bigr)^2\Big] \\
		&\lesssim 1.
	\end{align*}
	To bound the final term in \eqref{b-dnn} we simply recall that $\und\Ups_N$ has frequency support in $\{|n|\le M\}$ so that
	\begin{align*}
		\EE\Bigl[\frac12\int_0^1 \norm{\dot{\und\Ups}_N(t)}{H^\alpha}^2\, \rd t\Big] &\lesssim M^{2\alpha} \EE\norm{\und\Ups_N}{L^2}^2 \lesssim M^{3\alpha}
	\end{align*}
	as before. Compiling all of the above, we have, for $M$ sufficiently large depending on $\sigma$, and $\eps$ sufficiently small depending on $M$, that 
	\begin{equation*}
		-\log\EE[\exp(\delta\min\{\norm{\varphi_\eps\ast Y_N}{\cA}^q\}-\delta\norm{Y_N}{\cA}^q-V_N(Y_N))]\lesssim_{\sigma,\delta,A,\gamma} 1+M^{3\alpha}-|\sigma|M^{3\alpha};
	\end{equation*}
	therefore if $|\sigma|$ is sufficiently large, then the above tends to $-\infty$ as $M\to\infty$, proving the required divergence.
\end{proof}
\begin{prop} \label{noweaklimits}
	Let $d=3\alpha$. For $\sigma_1$ as in Proposition~\ref{nonnormalisability} and when $|\sigma|\ge\sigma_1$, the truncated measures $(\varrho_N)$ have no weak limit, even up to a subsequence.
\end{prop}
\begin{proof}
	Let
	\begin{equation}
		\vartheta^N_\delta(\rd u) = (\cZ^N_\delta)^{-1}\exp(-\delta\norm{u}{\cA}^q)\, \varrho_N(\rd u).
	\end{equation}
	We have the following alternate way to build $\vartheta_\delta$.
	\begin{lemma}\label{vartheta^N}
	As measures on $\cC^{\alpha-\frac d2-\eps}$ we have $\vartheta^N_\delta\rightharpoonup\vartheta_\delta$, and $\cZ^N_\delta\to\cZ_\delta$.
	\end{lemma} 
	\noindent Delaying the proof of Lemma~\ref{vartheta^N}, we now prove Proposition~\ref{noweaklimits}. Assume, for contradiction, that $\varrho_N\rightharpoonup\nu$. The observation
	\begin{align*}
		\vartheta_\delta(\rd u)&=\wlim\limits_{N\to\infty} \frac{\cZ_N}{\cZ^N_\delta}\exp(-\delta\norm{u}{\cA}^q)\, \varrho_N(\rd u) \\
		&=\frac{\cZ}{\cZ_\delta}\exp(-\delta\norm{u}{\cA}^q)\, \nu(\rd u)
	\end{align*}
	implies that $\varrho^\delta=\cZ\cZ_\delta^{-1}\nu$; since $\nu$ is a probability measure, this is a contradiction to Proposition~\ref{nonnormalisability} as it implies $\varrho^\delta(\cD')<\infty$.
\end{proof}
\begin{proof}[Proof of Lemma~\ref{vartheta^N}]
	We will first prove that $\cZ^N_\delta\to\cZ_\delta$, for which it suffices to show that $|\cZ^N_\delta-\cZ_{N,\delta}|\to 0$. To this end, compute
	\begin{align*}
		|\cZ^N_\delta-\cZ_{N,\delta}| &\le \int_{\cD'}|\exp(-\delta\norm{u}{\cA}^q-V_N(u))-\exp(-\delta\norm{u_N}{\cA}^q-V_N(u))|\, \mu(\rd u) \\
		&=\int_{\cD'} \re^{-\delta(\norm{u}{\cA}^q\wedge\norm{u_N}{\cA}^q)-V_N(u)}(1-\re^{-\delta|\norm{u}{\cA}^q-\norm{u_N}{\cA}^q|})\, \mu(\rd u) \\
		&\le \delta\int_{\cD'} \re^{-\delta(\norm{u}{\cA}^q\wedge\norm{u_N}{\cA}^q)-V_N(u)}|\norm{u}{\cA}^q-\norm{u_N}{\cA}^q|\, \mu(\rd u) \\
		&\lesssim \delta\int_{\cD'} \re^{-c\delta\norm{u_N}{\cA}^q-V_N(u)}|\norm{u}{\cA}-\norm{u_N}{\cA}|\cdot\norm{u}{\cA}^{q-1}\, \mu(\rd u) \\
		&\lesssim \delta\int_{\cD'} \re^{-c\delta\norm{u_N}{\cA}^q-V_N(u)}\norm{u-u_N}{\cA}\cdot\norm{u}{\cA}^{q-1}\, \mu(\rd u) \\
		&\lesssim \delta\int_{\cD'} \re^{-c\delta\norm{u_N}{\cA}^q-V_N(u)}\norm{u-u_N}{W^{-\alpha+\eps,3}}\cdot\norm{u}{\cA}^{q-1}\, \mu(\rd u) \\
		&\lesssim \delta N^{-a}\int_{\cD'} \re^{-c\delta\norm{u_N}{\cA}^q-V_N(u)}\norm{u}{W^{-\alpha+a+\eps,3}}\cdot\norm{u}{\cA}^{q-1}\, \mu(\rd u) \\
		&\lesssim \delta N^{-a} \cZ_{N,c\delta}\int_{\cD'} \norm{u}{W^{-\alpha+a+\eps,3}}\cdot\norm{u}{\cA}^{q-1}\, \vartheta_{N,c\delta}(\rd u) \\
	\end{align*}
	using the mean value theorem for the third line, the $\cA\to\cA$-boundedness of $\pi_N$ for the fourth line, \eqref{embedAintosobolev} for the sixth line, and the $L^3\to L^3$-boundedness of $1-\pi_N$ for the seventh line. Recalling that the $\cZ_{N,c\delta}$ are uniformly bounded, picking, e.g., $a=\frac\alpha4$ to permit $W^{-\alpha+a+\eps,3}\supseteq {\rm supp}\,\mu$, and using $r^k\lesssim \exp(\delta' r^\ell)$ for any $k,\ell$, leaves us with
	\begin{align*}
		|\cZ^N_\delta-\cZ_{N,\delta}| &\lesssim \delta N^{-\frac\alpha4}\int_{\cD'} \re^{\frac c2\delta\norm{u}{\cA}^q+\delta'\norm{u}{W^{-\frac\alpha2-\eps,\infty}}^2}\, \vartheta_{N,c\delta}(\rd u) \\
		&\lesssim \liminf_{K\to\infty}\delta N^{-\frac\alpha4} \int_{\cD'} \re^{\frac c2\delta\norm{u_K}{\cA}^q+\delta'\norm{u_K}{W^{-\frac\alpha2-\eps,\infty}}^2}\, \vartheta_{N,c\delta}(\rd u).
	\end{align*}
	One can now close the argument by the Bou\'e-Dupuis formula:
	\begin{align*}
		-\log\int_{\cD'} &\re^{\frac c2\delta\norm{u_K}{\cA}^q+\delta'\norm{u_K}{H^\alpha}^2}\, \vartheta_{N,c\delta}(\rd u) \\
		&= \inf_{\dot\Ups^K\in\HH_{\rm a}^\alpha}\EE\Bigl[-\frac c2\delta\norm{Y_K+\Theta_K}{\cA}^q-\delta'\norm{Y_K+\Theta_K}{W^{-\frac\alpha2-\eps,\infty}}^2 \\
		&\hskip90pt+c\delta\norm{Y_K+\Theta_K}{\cA}^q+V_N(Y_K+\Theta_K)+\frac12\int_0^1\norm{\dot\Ups^K(t)}{H^\alpha}\, \rd t\Bigr];
	\end{align*}
	where we write $Y_K=Y_N+(Y_K-Y_N)$ and $\Theta_K=\Theta_N+(\Theta_K-\Theta_N)$ and deal with tail terms separately.
\end{proof}

We conclude the this subsection with the proofs of Lemmas \ref{fMest} and \ref{ZMest}, which are essentially identical to those found in \cite[Lemmas 5.13 and 5.14]{OOT25}).
\begin{proof}[Proof of Lemma~\ref{fMest}]
	For \eqref{fML2} we use the Poisson summation formula to write
	\begin{equation} \label{fML2poisson} \begin{aligned}
		\int_{\TT^d} f_M^2\, \rd x &=M^d\Bigl(\int_{\TT^d} f(Mx)^2\, \rd x+\int_{\TT^d} f(Mx)\sum_{m\ne 0}f(Mx+Mm)\, \rd x \\
		&\hskip100pt+\int_{\TT^d} \sum_{m,m'\ne 0} f(Mx+Mm)f(Mx+Mm')\, \rd x\Bigr)
	\end{aligned} \end{equation}
	The first integral in \eqref{fML2poisson} can be estimated by using a change-of-variable, namely,
	\begin{align*}
		M^d\int_{\TT^d} f(Mx)^2\, \rd x&=M^d\int_{|y|\le M} f(y)^2M^{-d}\, \rd y \\
		&=1-\int_{|y|>M}f(y)^2\, \rd y \\
		&=1-O(M^{-c})
	\end{align*}
	for any $c>0$, using that $f$ is $L^2(\RR^d)$-normalised and its Schwartz decay. Moreover for $x\in\TT^d$, we can use this Schwartz decay to write
	\begin{equation*}
		|f(Mx+Mm)|\lesssim |Mm|^{-d-c}
	\end{equation*}
	so that the second and third integrals in \eqref{fML2poisson} are essentially bounded by
	\begin{equation*}
		\sum_{m\ne 0} |Mm|^{-d-c}+\sum_{m,m'\ne 0} |Mm|^{-d-c}|Mm'|^{-d-c}\lesssim M^{-d-c},
	\end{equation*}
	which is enough for \eqref{fML2}. For \eqref{fMH-c}, we use Plancherel's theorem and the boundedness of $\hat f$ to write
	\begin{align*}
		\int_{\TT^d} (\ja{\nabla}^{-c}f_M)^2\, \rd x &= \sum_{n\in\ZZ^d}\frac{|\hat f_M(n)|^2}{\ja{n}^{2c}} \\
		&= M^{-d}\sum_{\frac M2<|n|\le M}\frac{|\hat f(\frac nM)|^2}{\ja{\nabla}^{2c}} \\
		&\lesssim M^{-d-2c}\sum_{\frac M2<|n|\le M}\Big|\hat f\Bigl(\frac nM\Bigr)\Big|^2 \\
		&\lesssim M^{-2c}.
	\end{align*}
	For \eqref{fML3}, first compute
	\begin{align*}
		\int_{\TT^d} f_M^3\, \rd x &= \int_{\TT^d} \sum_{\frac M2<|n_1|,|n_2|,|n_3|\le M}M^{-\frac{3d}2}\hat f\Bigl(\frac{n_1}M\Bigr)\hat f\Bigl(\frac{n_2}M\Bigr)\hat f\Bigl(\frac{n_3}M\Bigr)\re^{2\pi\ri(n_1+n_2+n_3)\cdot x}\, \rd x \\
		&=\sum_{\frac M2<|n_1|,|n_2|\le M}M^{-\frac{3d}2}\hat f\Bigl(\frac{n_1}M\Bigr)\hat f\Bigl(\frac{n_2}M\Bigr)\hat f\Bigl(-\frac{n_1+n_2}M\Bigr) \\
		&\sim M^{\frac d2}.
	\end{align*}
	From the above one has the lower bound on $\norm{f_M}{L^3}^3$. For the upper bound, by the Hausdorff-Young inequality, and using the support and boundedness of $\hat f$, we have
	\begin{align*}
		\int_{\TT^d} |f_M|^3\, \rd x&\le \norm{\hat f_M}{\ell^{\frac32}}^3 \\
		&=\Bigl(\sum_{\frac M2<|n|\le M} \Big|M^{-\frac d2}\hat f\Bigl(\frac nM\Bigr)\Big|^{\frac32}\Bigr)^3 \\
		&\lesssim M^{\frac d2},
	\end{align*}
	which completes the proof of \eqref{fML3} and so that of Lemma~\ref{fMest}.
\end{proof}
\begin{proof}[Proof of Lemma~\ref{ZMest}]
	The proof of \eqref{kappaM} is the following computation: 
	\begin{align*}
		\kappa_M &=\sum_{|n|,|m|\le M}\frac{\EE[B_n(\frac12)B_m(\frac12)]}{\ja{n}^\alpha\ja{m}^\alpha}\re^{2\pi\ri (n+m)\cdot x} \\
		&\sim \sum_{|n|\le M} \ja{n}^{-2\alpha} \\
		&\sim M^{d-2\alpha}.
	\end{align*}
	For \eqref{ZMLp}, use Fubini's theorem and the Wiener chaos estimate as
	\begin{align*}
		\EE\Bigl[\int_{\TT^d} |Z_M|^p\, \rd x\Big] &= \int_{\TT^d} \EE|Z_M(x)|^p\, \rd x \\
		&\lesssim_p \int_{\TT^d} (\EE[Z_M(x)^2])^{\frac p2}\, \rd x \\
		&\sim M^{\frac p2(d-2\alpha)}.
	\end{align*} 
	To prove \eqref{ZMconst}, we observe
	\begin{align*}
		\EE&\Bigl[\Bigl(\int_{\TT^d} Z_M^2\, \rd x-\kappa_M\Bigr)^2\Big]  \\
		&= \EE\Bigl[\Bigl(\int_{\TT^d}\sum_{|n|,|m|\le M}\frac{B_n(\frac12)B_m(\frac12)-\EE[B_n(\frac12)B_m(\frac12)]}{\ja{n}^\alpha\ja{m}^\alpha}\re^{2\pi\ri (n+m)\cdot x}\, \rd x\Bigr)^2\Big] \\
		&=\EE\Bigl[\Bigl(\sum_{|n|\le M}\frac{|B_n(\frac12)|^2-\frac12}{\ja{n}^{2\alpha}}\Bigr)^2\Big] \\
		&=\sum_{|n|,|m|\le M}\frac{\EE[(|B_n(\frac12)|^2-\frac12)(|B_m(\frac12)|^2-\frac12)]}{\ja{n}^{2\alpha}\ja{m}^{2\alpha}} \\
		&=\sum_{|n|\le M}\frac{\EE(|B_n(\frac12)|^2-\frac12)^2}{\ja{n}^{4\alpha}} \\
		&\lesssim 1+M^{d-4\alpha} \\
		&\lesssim 1
	\end{align*}
	and, analogously to above using the independence of $B_n(\frac12)$ from $B_n(1)-B_n(\frac12)$,
	\begin{align*}
		\EE&\Bigl[\Bigl(\int_{\TT^d} Y_NZ_M\, \rd x-\int_{\TT^d} Z_M^2\, \rd x\Bigr)^2\Big] \\
		&= \EE\Bigl[\Bigl(\int_{\TT^d}\sum_{|n|,|m|\le M}\frac{B_n(\frac12)(B_m(1)-B_m(\frac12))}{\ja{n}^{2\alpha}\ja{m}^{2\alpha}}\re^{2\pi\ri(n+m)\cdot x}\, \rd x\Bigr)^2\Big] \\
		&= \sum_{|n|,|m|\le M}\frac{\EE[B_n(\frac12)\bar{B_n(1)-B_n(\frac12)}B_m(\frac12)\bar{B_m(1)-B_m(\frac12)}]}{\ja{n}^{2\alpha}\ja{m}^{2\alpha}} \\
		&\lesssim 1.
	\end{align*}
	Finally, for \eqref{ZMsmall}, we compute
	\begin{align*}
		\EE\Bigl[\Bigl(\int_{\TT^d} Y_Nf_M\, \rd x\Bigr)^2\Big] &= \EE\Bigl[\Bigl(\sum_{\frac M2<|n|\le M}\hat Y_N(n)\hat f_M(n)\Bigr)^2\Big] \\
		&=M^{-d}\sum_{\frac M2<|n|,|m|\le M} \frac{\EE[B_n(1)B_m(1)]}{\ja{n}^\alpha\ja{m}^\alpha}\hat f\Bigl(\frac nM\Bigr)\hat f\Bigl(\frac mM\Bigr) \\
		&M^{-d}\sum_{\frac M2<|n|\le M}\frac1{\ja{n}^{2\alpha}}\hat f\Bigl(\frac nM\Bigr)^2 \\
		&\lesssim M^{-2\alpha}
	\end{align*}
	and similarly for $\EE[(\int_{\TT^d} Z_Mf_M\, \rd x)^2]$, which completes the proof of Lemma~\ref{ZMest}.
\end{proof}
%

\section*{Acknowledgements}
The author would like to thank his supervisor Leonardo Tolomeo for his suggestions and comments throughout the completion of this project.

\printbibliography

\end{document}